\documentclass[oneside,english,12pt]{amsart}
\usepackage{lmodern}
\usepackage[T1]{fontenc}       			 
\usepackage[latin9]{inputenc}  			 
\usepackage[english]{babel}			 
\usepackage[top=3cm, bottom=3cm, left=3cm, right=3cm, heightrounded, marginparwidth=2.5cm, marginparsep=6mm]{geometry}

\usepackage{soul}
\usepackage{setspace} 
\usepackage{marginnote}
\usepackage{verbatim}
\usepackage{float}
\usepackage{mathrsfs}	
\usepackage{amstext}
\usepackage{amsthm}
\usepackage{amssymb}
\usepackage{amsopn} 
\usepackage{bbold}
\usepackage{stix}
\usepackage{enumitem}				
\usepackage[all]{xy}    				
\usepackage{aliascnt}   				
\usepackage{tikz}  \usetikzlibrary{matrix}	
\usepackage{hyperref}
\usepackage{mathdots}
\hypersetup{
    colorlinks,
    linkcolor={red!50!black},
    citecolor={green!50!black},
    urlcolor={blue!80!black},
    linktocpage
}
\usepackage{footnote}

\numberwithin{equation}{section}
\numberwithin{figure}{section}

\theoremstyle{theorem}
  \newtheorem*{cor*}{Corollary}
  \newtheorem*{thm*}{Theorem}
  \newtheorem*{lem*}{Lemma}
  \newtheorem*{claim*}{Claim}
  \newtheorem*{mclaim*}{Main Claim}
  
  \newtheorem{thmx}{Theorem}

  \newaliascnt{corx}{thmx}
  \newtheorem{corx}[corx]{Corollary}
  \aliascntresetthe{corx}

  \newtheorem{thm}{Theorem}[section]

  \newaliascnt{lem}{thm}
  \newtheorem{lem}[lem]{Lemma}
  \aliascntresetthe{lem}

  \newaliascnt{klem}{thm}
  
  \aliascntresetthe{klem}

  \newaliascnt{cor}{thm}
  \newtheorem{cor}[cor]{Corollary}
  \aliascntresetthe{cor}

  \newaliascnt{prop}{thm}  
  \newtheorem{prop}[prop]{Proposition}
  \aliascntresetthe{prop}

\theoremstyle{definition}
  \newaliascnt{defn}{thm}
  \newtheorem{defn}[defn]{Definition}
  \aliascntresetthe{defn}

  \newaliascnt{exmpl}{thm}
  \newtheorem{exmpl}[exmpl]{Example}
  \aliascntresetthe{exmpl}

\theoremstyle{remark}
  \newaliascnt{rem}{thm}
  \newtheorem{rem}[rem]{Remark}
  \aliascntresetthe{rem}

  \theoremstyle{remark}
  \newaliascnt{set}{thm}
  \newtheorem{set}[set]{Setting}
  \aliascntresetthe{set}

  \theoremstyle{remark}
  \newaliascnt{con}{thm}
  \newtheorem{con}[con]{Construction}
  \aliascntresetthe{con}

\newcommand{\bbC}{\mathbb{C}}

\newcommand{\bbH}{\mathbb{H}}

\newcommand{\bbK}{\mathbb{K}}

\newcommand{\bbN}{\mathbb{N}}

\newcommand{\bbQ}{\mathbb{Q}}
\newcommand{\bbR}{\mathbb{R}}

\newcommand{\clC}{\mathcal{C}}

\DeclareMathOperator{\dd}{d}
\DeclareMathOperator{\EE}{E}

\DeclareMathOperator{\coker}{coker}

\DeclareMathOperator{\HH}{H}

\DeclareMathOperator{\im}{im}
\DeclareMathOperator{\Ind}{Ind}
\DeclareMathOperator{\Int}{Int}

\DeclareMathOperator{\stab}{Stab}

\newcommand{\Hb}{{\rm H}_{\rm b}}

\newcommand{\Linfty}{L^\infty}

\newcommand{\IE}{{}^{\rm I}{\rm E}}
\newcommand{\Id}{{}^{\rm I}{\rm d}}
\newcommand{\IIE}{{}^{\rm II}{\rm E}}
\newcommand{\IId}{{}^{\rm II}{\rm d}}

\DeclareMathOperator{\GL}{GL}         
\DeclareMathOperator{\SL}{SL}         
\DeclareMathOperator{\OO}{O}               
\DeclareMathOperator{\UU}{U}               
\DeclareMathOperator{\Sp}{Sp}              

\newcommand{\qand}{\quad \mathrm{and} \quad}

\AtBeginDocument{\addtocontents{toc}{\protect\setlength{\parskip}{0pt}}}

\begin{document}

\title[]{A Quillen stability criterion for bounded cohomology}
\author{Carlos De la Cruz Mengual}
\address{Faculty of Electrical and Computer Engineering \\ Technion, Haifa, Israel}
\email{c.delacruz@technion.ac.il}

\author{Tobias Hartnick}
\address{Institute of Algebra and Geometry \\
KIT, Karlsruhe, Germany}
\email{tobias.hartnick@kit.edu}

\begin{abstract}
We provide a version of Quillen's homological stability criterion for continuous bounded cohomology. This criterion is exploited in the companion paper \cite{DM+Hartnick} in order to derive new bounded cohomological stability results for various families of classical groups.
\end{abstract}

\maketitle

\section{Introduction}
An infinite chain
$
	(G_r)_{r\geq 0} = (G_0 < G_1 < G_2 < \cdots)
$ 
of groups is called \emph{homologically stable} if there exists a function $r: \bbN \to \bbN$ such that the respective inclusions induce isomorphisms
\[
\HH_q(G_{r(q)}) \ \cong \ \HH_q(G_{r(q)+1}) \ \cong \ \HH_q(G_{r(q)+2}) \ \cong \ \cdots   
\]
in group homology for all $q \in \bbN$.\footnote{As usual in homological algebra, we enumerate starting from $0$. Accordingly, we
convene that $0 \in \bbN$, and for any $k \in \bbN$, we denote by $[k]$ the $(k+1)$-element set $\{0,\ldots,k\} \subset \bbN$. We also set $[\infty] := \bbN$.} Any such function $r$ is then called a \emph{stability range} for the family. Homological stability (with a linear stability range) has been established for many families of classical linear algebraic groups, but also for various families of non-linear groups of interest such as mapping class groups or automorphism groups of free groups. 

A particularly succesful method to establish homological stability, employed for example by van der Kallen \cite{vdK}, Harer \cite{Harer}, Hatcher and Vogtmann \cite{HatVogt}, Essert \cite{Essert}, and Sprehn and Wahl \cite{Sprehn-Wahl2}, is based on \emph{Quillen's stability criterion}, which can be stated as follows: Suppose that, for every $r \in \bbN$, we are given a $\Delta$-complex $X(r)$ (also known as semi-simplicial set) endowed with a simplicial $G_r$-action,
and two natural numbers $\gamma(r), \tau(r)$ with the following properties:
\begin{enumerate}[label=(Q\arabic*),leftmargin=2.3pc]
	\item $X(r)$ is \emph{$\gamma(r)$-acyclic}, i.e. the reduced homology $\tilde{\HH}_\bullet(X(r))$ vanishes up to degree $\gamma(r)$;
	\item $X(r)$ is \emph{$\tau(r)$-transitive}, i.e. there is only one $G_r$-orbit of $l$-simplices for $l \in \{0,\ldots,\tau(r)\}$; 
	\item the complexes are \emph{$\tau(r)$-compatible}, i.e.\ the stabilizer of an $l$-simplex in $X(r)$ is isomorphic to $G_{r-l-1}$ for every $l < \tau(r)$, and these isomorphisms are compatible with inclusions of stabilizers in the sense of Condition (MQ3a) below.
\end{enumerate} 
Then $(G_r)_{r \geq 0}$ is homologically stable provided $\min\{\gamma(r), \tau(r)\} \to \infty$ as $r \to \infty$. In addition, a stability range can be computed explicitly from the functions $\gamma$ and $\tau$. See Quillen's unpublished notes \cite{Quillen}, or the more recent treatments by Bestvina \cite{Bestvina} and Sprehn and Wahl \cite{Sprehn-Wahl2} for details and further references. 

The purpose of the present article is to establish a similar stability criterion in \emph{continuous bounded cohomology}. The theory of bounded cohomology for discrete groups goes back to Johnson, Trauber, and Gromov \cite{Gromov}, and was later extended to the realm of topological groups by Burger and Monod \cite{Burger-Monod1}; we refer to the monographs by Frigerio \cite{Frigerio} and Monod \cite{Monod-Book} for background and numerous applications. 

By analogy with the classical situation, we say that an infinite family $G_0<G_1< G_2 < \cdots$ of locally compact groups is \emph{bc-stable} (short for \emph{bounded-cohomologically stable}) with \emph{stability range} $r: \bbN \to \bbN$ if, for every $q \geq 0$, the respective inclusions induce isomorphisms
\[
\Hb^q(G_{r(q)}) \ \cong \ \Hb^q(G_{r(q)+1}) \ \cong \ \Hb^q(G_{r(q)+2}) \ \cong \ \cdots   
\]
in continuous bounded cohomology. We will give conditions analogous to (Q1) -- (Q3) and show that these imply bc-stability. For technical reasons we will assume that $G$ is either second-countable or discrete. This setting is broad enough to deal with both almost connected Lie groups (as considered in our article \cite{DM+Hartnick}) and with uncountable discrete groups (as considered by Monod and Nariman \cite{Monod-Nariman}).

Compared to classical (co-)homology, (continuous) bounded cohomology has a more functional analytic flavour---in fact, it can be defined as a derived functor in some suitable exact category of Banach modules; see work of B\"uhler \cite{Buehler}. While it is sometimes possible to extend classical cohomological results to this setting, this usually requires additional efforts in order to take certain functional analytic peculiarities into account. These peculiarities are the reason why we need to assume second countability in the non-discrete case.

\subsection{General setting}
Throughout this article let $G$ denote a locally compact group. We will assume throughout that $G$ is either second-countable or discrete. To formulate analogues of Quillen's conditions in the setting of bounded cohomology, we need to consider $G$-actions on complexes of a suitable type. 
\begin{itemize}
\item In case $G$ is discrete, the associated complexes will be, as in standard (co)homology, $\Delta$-complexes with a simplicial $G$-action (i.e.\ semi-simplicial $G$-sets); their sets of simplices will be endowed with the counting measure.
\item In case $G$ is second-countable and non-discrete, we will need to replace $\Delta$-complexes by a class that is better adapted to a measurable setting, namely, semi-simplicial objects in the category of \emph{Lebesgue $G$-spaces}. Here, by a Lebesgue $G$-space we mean a standard Borel space with a Borel $G$-action and a $G$-invariant probability measure class.
\end{itemize}
In order to unify the treatment of the two cases we will refer to the corresponding complexes as \emph{Lebesgue $G$-complexes} in either case. To a Lebesgue $G$-complex $X = (X_q)_{q \geq 0}$, we can associate a complex
\[
0 \to \mathbb R \to \Linfty(X_0) \to \Linfty(X_1) \to \Linfty(X_2) \to \cdots
\]
of Banach spaces, where the codifferentials are given by alternating sums of dual face maps. Note that in the discrete case, this complex reduces to
\[
0 \to \mathbb R \to \ell^\infty(X_0) \to  \ell^\infty(X_1) \to  \ell^\infty(X_2) \to \cdots
\]
We then say that $X$ is \emph{boundedly $\gamma_0$-acyclic} if the cohomology of this complex vanishes up to degree $\gamma_0$. We also say that $X$ is \emph{essentially $\tau_0$-transitive} if $G$ acts essentially transitively (i.e.\ with a co-null orbit) on $q$-simplices for $q \in \{0, 1, \dots, \tau_0\}$. In this case, we can choose simplices $o_0 \in X_0, \dots, o_{\tau_0} \in X_{\tau_0}$ such that $o_i$ is a face of $o_{i+1}$ for all $i \in \{0, \dots, \tau_0-1\}$, and such that the orbits of $o_0, \dots, o_{\tau_0}$ are co-null in the respective spaces of simplices (see Construction \ref{ConGenericFlag}). We then refer to $(o_0, \dots, o_{\tau_0})$ as a \emph{generic flag} in $X$.

\subsection{A spectral sequence in bounded cohomology}
Our main tool will be the following spectral sequence which connects the continuous bounded cohomology ring of a locally compact group $G$ (always assumed to be second-countable or discrete) to those of the stabilizers of simplices of a generic flag in a highly boundedly acyclic and highly essentially transitive Lebesgue-$G$-complex.
\begin{thmx}\label{BasicSS} Let $X$ be a Lebesgue $G$-complex, which is boundedly $\gamma_0$-acyclic and essentially $\tau_0$-transitive for some $\gamma_0, \tau_0 \in \mathbb N \cup \{\infty\}$. Moreover, let  $(o_{0}, \dots, o_{\tau_0})$ be a generic flag in $X$ with stabilizers $H_i := \stab_G(o_i)$ for $i \in \{0, \dots, \tau_0\}$, and set $H_{-1} := G$. 
\begin{enumerate}[leftmargin=3.6pc, label=\emph{(\arabic*)}]
\item For every $p \in \{0,\ldots,\tau_0-1\}$ and $i \in \{0,\ldots,p+1\}$, there exists an element $w_{p,i} \in G$ such that $\delta_i(o_{p+1})= w_{p,i}^{-1} \cdot o_{p}$, that induces an injective homomorphism \vspace{-5pt}
\[
\Int(w_{p,i}): H_{p+1} \to H_{p}, \quad h \mapsto w_{p, i}h w_{p, i}^{-1}. 
\]
\item There exists a first-quadrant spectral sequence $\EE_\bullet^{\bullet,\bullet}$ such that:
\begin{enumerate}[label = \emph{(\roman*)},leftmargin=25pt]
\item $\EE_\infty^t = 0$ for every $t \in \{0, \dots, \gamma_0 + 1\}$.
\item $\EE_1^{p,q} = \Hb^q(H_{p-1})$ for all $p \in \{0, \dots \tau_0+1\}$ and $q\geq 0$.
\item $\EE_2^{p,0} = 0 \; \text{for all }p \in \{0, \dots \tau_0+1\}$.
\item The first page differentials $\dd_1^{0,q}:  \Hb^q(G) \to \Hb^q(H_{0})$ are induced by the inclusion $H_{0} \hookrightarrow G$ for all $q \geq 0$.
\item The first page differentials $\dd_1^{p,q}:  \Hb^q(H_{p-1}) \to \Hb^q(H_{p})$ are the alternating sums \vspace{-10pt}
\[
\dd_1^{p,q} = {\textstyle \sum_{i=0}^{p}} (-1)^i \cdot \Hb^q(\Int(w_{p-1, i})).
\]
for all $p \in \{1, \dots, \tau_0\}$, $q\geq 0$, and any choice of elements $w_{p,i} \in G$ as in \emph{(1)}.
\end{enumerate}
\end{enumerate}
\end{thmx}

\begin{rem}
A version of this spectral sequence for product complexes (i.e.\ complexes where $X_r = X_0^{r+1}$ with the usual differential) of locally compact second-countable groups was first studied by Monod in \cite{Monod-Stab}. The version for uncountable discrete groups played a key role in work by Monod and Nariman \cite{Monod-Nariman} (see \cite[Theorem 3]{Monod-Nariman}) on the computation of the bounded cohomology of various homeomorphism and diffeomorphism groups of manifolds. 
\end{rem}

\subsection{Measurable Quillen families}
Assume that we are given an infinite family $(G_r)_{r\geq 0}=(G_0 < G_1 < G_2 < \cdots)$ of locally compact groups which are either all second-countable or all discrete and
a Lebesgue $G_r$-complex $X(r)$ for every $r \geq 0$. We can then define counterparts to the Quillen conditions (Q1) and (Q2) as follows. We assume that we are given functions $\gamma, \tau: \mathbb N \to \mathbb N \cup \{\pm \infty\}$ such that
\begin{enumerate}[leftmargin=3.6pc, label=(MQ\arabic*)]
	\item $X(r)$ is boundedly $\gamma(r)$-acyclic.
	\item $X(r)$ is essentially $\tau(r)$-transitive.
\end{enumerate} 
These assumptions imply by Theorem \ref{BasicSS} that, for each $r \geq 0$, we obtain elements $w_{r,p,i} \in G_r$ for all $p \in \{0,\ldots,\tau(r)\}$ and $i \in \{0,\ldots,p+1\}$ as in Part (1) of Theorem \ref{BasicSS}, and a spectral sequence as in (2), relating the continuous bounded cohomology of $G_r$ to the continuous bounded cohomology of the stabilizers $G_r=: H_{r,-1} > H_{r,0} > \dots >H_{r,\tau(r)}$ of a generic flag $(o_{r,0}, \dots, o_{r, \tau(r)})$ in $X(r)$. To obtain a stability result in the vein of Quillen's, we need to relate the bounded cohomology of these stabilizers to the one of the previous groups in the sequence. We offer three different compatibility conditions which are useful in different situations and are all sufficient to obtain bc-stability:

\begin{enumerate}[leftmargin=3.7pc, label=(MQ3\alph*)]
\item For all $r \geq 0$ and $p \in \{-1,0, \dots, \tau(r)\}$, we have $H_{r,p} = G_{r-p-1}$  (we convene that $G_{k}$ is the trivial group for $k<0$), and the conjugation ${\rm Int}(w_{r,p,i}): H_{r,p+1} \to H_{r,p}$ equals the inclusion $\iota_{r-p-2}:G_{r-p-2} \hookrightarrow G_{r-p-1}$ for all $r$, all $p \leq \tau(r)-1$ and all $i$. 

\item For all $r \geq 0$ and $p \in \{-1,0, \dots, \tau(r)\}$, there is an epimorphism
$
	\pi_{r,p}: H_{r,p} \twoheadrightarrow G_{r-p-1} 
$
with amenable kernel such that \vspace{-7pt}
\begin{equation*}
		\begin{gathered}	
		\xymatrixcolsep{3.5pc}
		\xymatrix@R=13pt{H_{r,p+1} \ar@{^{(}->}[r]^{{\rm Int}(w_{r,p,i})}  \ar@{->>}[d]_{\pi_{r,p+1}} & H_{r,p} \ar@{->>}[d]^{\pi_{r,p}} \\
		G_{r-p-2} \ar@{^{(}->}[r]^{\iota_{r-p-2}} & G_{r-p-1}} \vspace{-5pt}
		\end{gathered}
\end{equation*}
is a commutative diagram for all $r,$ all $p \leq \tau(r)-1$, and all $i$.  

\item For all $r \geq 0$ and $p \in \{-1,0, \dots, \tau(r)\}$, there is an epimorphism
$
	\pi_{r,p}: H_{r,p} \twoheadrightarrow G_{r-p-1} 
$
with amenable kernel and a continuous homomorphic section $\sigma_{r,p}$ of $\pi_{r,p}$ such that \vspace{-7pt}
\begin{equation*}
		\begin{gathered}	
		\xymatrixcolsep{3.5pc}
		\xymatrix@R=13pt{H_{r,p+1} \ar@{^{(}->}[r]^{{{\rm Int}(w_{r,p,i})}} & H_{r,p} \ar@{->>}[d]^{\pi_{r,p}} \\
		G_{r-p-2} \ar[u]^{\sigma_{r,p+1}} \ar@{^{(}->}[r]^{\iota_{r-p-2}} & G_{r-p-1}} 
		\end{gathered}
\end{equation*}
is a commutative diagram for all $r$, all $p \leq \tau(r)-1$, and all $i$.
\end{enumerate}
Here, (MQ3a) is Quillen's original condition: indeed, the equality ${\rm Int}(w_{r,p,i})=\iota_{r-p-2}$ means that the element $w_{r,p,i} \in G_r$ centralizes $G_{r-p-2} < G_r$. The conditions (MQ3b) and (MQ3c) are two different relaxations, which take advantage of the fact that continuous surjections with amenable kernel induce isomorphisms in continuous bounded cohomology. The latter is the most technical, but also most useful condition, as witnessed by Example \ref{GLnExample} below.

\begin{defn}\label{DefInfiniteQuillen} We say that $(G_r,X(r))_{r \geq 0}$ is a \emph{measurable $(\gamma, \tau)$-Quillen family} if conditions (MQ1) and (MQ2), and either (MQ3a), (MQ3b) or (MQ3c) hold.
\end{defn}
The following theorem is the main qualitative result of this article:
\begin{thmx}[Quillen stability for continuous bounded cohomology]\label{MainThmQual} Assume that  $(G_r,X(r))_{r \geq 0}$ is a measurable $(\gamma, \tau)$-Quillen family. If both $\gamma$ and $\tau$ are proper, then $(G_r)_{r \geq 0}$ is bc-stable.
\end{thmx}
Theorem \ref{MainThmQual} will be established by comparing the various spectral sequences associated to the Lebesgue complexes $X(r)$ by means of Theorem \ref{BasicSS}. This will actually provide a quantitative version of Theorem \ref{MainThmQual}: we will obtain an explicit stability range which depends only on the functions $\gamma$ and $\tau$. This quantitative version actually also makes sense for finite Quillen families. We will state the most general version of our main theorem as Theorem \ref{MainTheorem} below.

For \emph{product complexes}, i.e. complexes of the form $X(r)_q = X(r)_0^{q+1}$ with the forgetful face maps, Theorem \ref{MainThmQual} and its quantitative version were essentially established by Monod in \cite{Monod-Stab}, and our proof is an extension of Monod's original proof. As we will explain after Example \ref{GLnExample}, our more general version is crucial if one wants to generalize Monod's results concerning bc-stability of general linear groups to other classes of classical groups.

\begin{rem}
The isomorphisms produced by Theorem \ref{MainThmQual} are generally not isometric with respect to the Gromov norms. This fact is evidenced, for instance, by the norms of the bounded Borel classes in complex simple Lie groups of classical type, which were computed in works by Bucher, Burger, and Iozzi \cite{BBI} and the first author \cite{DM-Thesis,DM-Complex}. 
\end{rem}

\subsection{The discrete case}
Recall that Theorem \ref{MainThmQual} applies to the case in which $(G_r)_{r \geq 0}$ are discrete groups. In this case, the statement simplifies considerably, and we record this special case for ease of reference. The sequence $(X(r))_{r \geq 0}$ is just a family of $\Delta$-complexes, and the sets of simplices of $X(r)$ are all endowed with the respective counting measures. To prove rational homological stability of $(G_r)_{r \geq 0}$ using Quillen's criterion, it suffices to show that $G_r$ acts $\tau(r)$-transitively on $X(r)$ with stabilizers $H_{r,p} = G_{r-p-1}$ and that $X(r)$ is rationally $\gamma(r)$-acyclic for two proper functions $\tau, \, \gamma: \bbN \to \bbN \cup \{\pm \infty\}$. One way to do this is by exhibiting a partial contracting chain homotopy $h_*$ for each of the complexes
\[
\dots \to \bbQ X(r)_1 \xrightarrow{d_0} \bbQ X(r)_0 \xrightarrow{d_{-1}} \bbQ \to 0
\]
so that $d_k \circ h_k + h_{k-1} \circ d_{k-1}= 1$. These chain homotopies will then induce a partial dual contracting chain homotopy of the dual cocomplexes
\[
0 \to \bbR \xrightarrow{d^{-1}} \ell^\infty(X(r)_0) \xrightarrow{d^0} \ell^\infty(X(r)_1) \to \dots
\]
provided that for every $k \in \bbN$ there is a constant $C_k > 0$ such that for
every simplex $\sigma \in X_k$ the $\ell^1$-norm $\|h_k(\sigma)\|_1$ is bounded by $C_k$, i.e. 
\[
h_k(\sigma) = \sum_{i=1}^{n_\sigma} \alpha_i \tau_i \implies  \sum_{i=1}^{n_\sigma} |\alpha_i| \leq C_k.
\]
In this case we refer to $h_*$ as a \emph{rational $\ell^1$-homotopy}.
\begin{corx}\label{Discrete} Let $(G_r)_{r\geq 0}$ be a family of discrete groups. For every $r \geq 0$, let $X(r)$ be a $\gamma(r)$-acyclic $\Delta$-complex with a $\tau(r)$-transitive $G_r$-action and stabilizers $H_{r,q} = G_{r-q-1}$ which are compatible in the sense of (MQ3a).
\begin{enumerate}[label = \emph{(\roman*)},leftmargin=25pt]
\item If $\gamma, \, \tau: \bbN \to \bbN \cup \{\pm \infty\}$ are proper, then $(G_r)_{r \geq 0}$ is homologically stable.
\item If, moreover, rational $\gamma(r)$-acyclicity of $X(r)$ is witnessed by a rational $\ell^1$-homotopy, then $(G_r)_{r \geq 0}$ is also bc-stable. 
\end{enumerate}
\end{corx}
Again, the stability range can be computed explicitly in terms of $\gamma$ and $\tau$; see Theorem \ref{MainTheorem}. Corollary \ref{Discrete} provides a clear strategy to establish bc-stability for all families of countable groups for which homological stability has been established using Quillen's method: One just has to check whether the underlying homotopy is rationally $\ell^1$. Unfortunately, the homotopies in question are usually only given implicitly. We thus leave it to future work to determine for which countable groups an explicit rational $\ell^1$-homotopy can be constructed.

\subsection{The case of reductive Lie groups}
Our main interest in establishing Theorem \ref{MainThmQual} in its present generality was to provide a framework for establishing bc-stability of the classical families of reductive Lie groups. For all of these families, bc-stability is predicted by the (notoriously open) isomorphism conjecture in continuous bounded cohomology. However, prior to this work, this predicted stability had only been established for general and special linear groups over $\bbR$ and $\bbC$ by Monod \cite{Monod-Stab} using a special case of Theorem \ref{MainThmQual}. Let us briefly explain how the results for general linear groups\footnote{As explained by Monod in the Note before Lemma 10 in \cite{Monod-Vanish}, the stability range given in \cite{Monod-Stab} is not correct due to an inaccuracy in the induction step. We use this opportunity to state the correct range; see \cite[Sect. 6.3 and 6.4]{DM-Thesis} for details. Also, the quaternionic case is absent from \cite{Monod-Stab}, but it does not require any new ideas.} fit into our present context:
\begin{exmpl}[General linear groups]\label{GLnExample} Let $G_r := \mathrm{GL}_r(\mathbb K)$ with $\bbK \in \{\bbR, \bbC, \bbH\}$. These groups constitute an infinite family $(G_0 < G_1 < G_2 < \cdots )$ of lcsc groups with the block inclusions
\begin{equation} \label{GLn_inclusion}
	\iota_r: G_r \hookrightarrow G_{r+1}, \quad A \mapsto 1 \times A := \begin{pmatrix} 1 & 0 \\ 0 & A \end{pmatrix}.
\end{equation}
Let $X(r)_q := \mathbb P(\mathbb K^r)^{q+1}$, where $\mathbb P(\mathbb K^r)$ denotes the projective space of $\mathbb K^r$ and the face maps are the usual forgetful maps. If we equip each of these spaces with its canonical Lebesgue measure class, then each $X(r)$ becomes a boundedly acyclic Lebesgue $G(r)$-complex. 
Moreover, the action of $G_r$ on $X(r)$ is essentially $r$-transitive. If $e_1, \dots, e_r$ denotes the standard basis of $\bbK^r$ and $e_{r+1} = \sum_{i=1}^r e_i$, then a generic flag in $X(r)$ is given by $(o_0, \dots,o_{r})$, where $o_q := ([e_1], \dots, [e_{q+1}])$. The corresponding stabilizers $H_{r,q} = \stab_{G_r}(o_q)$ are 
\[
H_{r,q} = \left\{\begin{pmatrix} D & V \\ 0 & A \end{pmatrix} \ \Bigg| \ D \in G_{q+1} \mbox{ diagonal}, \ A \in G_{r-q-1}, \ V \in \mathrm{M}_{(q+1) \times (r-q-1)}(\bbK) \right\}
\]
for $q \in [r-1]$, and $H_{r,r} = (\bbK^\times) \cdot I_r$; we also set $H_{r,-1} = G_r$. For $q \in \{0,\ldots,r-2\}$ and $i \in \{0,\ldots,q+1\}$, we let $w_{r,q,i} \in G_r$ be the matrix that permutes the standard basis vectors via the cycle $(q+2,q+1,\ldots,i+1)$, and for $q = r -1$, we fix an arbitrary choice of $w_{r,r-1,i} \in G_r$ according to Theorem \ref{BasicSS} (1). With this notation, condition (MQ3a) is violated, and while the obvious projections $\pi_{r,q}: H_{r,q} \to G_{r-q-1}$ do have amenable kernel, they do not satisfy Condition (MQ3b) since the corresponding diagram does not commute. In turn, Condition (MQ3c) is satisfied with the homomorphic sections \vspace{-3pt}
\[
	\sigma_{r,q}:G_{r-q-1} \to H_{r,q}, \quad A \mapsto I_{q+1} \times A, \vspace{-3pt}
\]
of $\pi_{r,q}$. This is the reason why we insist on this slightly technical compatibility condition. 

We may now deduce that $(G_r, X(r))_{r\geq 0}$ is a measurable $(\infty, r)$-Quillen family, hence $(G_r)$ is bc-stable. An explicit stability range can be read off from  Theorem \ref{MainTheorem} below. In the notation of that theorem, we have $q_0 = 2$, $\gamma(r) = \infty$ and $\tau(r) = r$ (since $\Hb^2(G_r) = 0$ for all $r \geq 0$; this follows from \cite{Burger-Monod3} and \cite[Cor. 8.8.6, Ex. 9.9.3]{Monod-Book}), and hence\vspace{-3pt}
\[
	\min\{\widetilde{\gamma}(q,r), \widetilde{\tau}(q,r)-1\} = \min_{j=q_0}^{q} \big\{\tau\big(r+1-2(q-j)\big) - j\big\}-1= r-2q+q_0 = r-(2q-2).\vspace{-3pt}
\]
We deduce that for every $q \geq q_0+1 = 3$, the inclusions $\iota_r$ induce isomorphisms and injections
\[
\dots \cong \Hb^q(\GL_{2q-1}(\bbK)) \cong \Hb^{q}(\GL_{2q-2}(\bbK)) \hookrightarrow \Hb^{q}(\GL_{2q-3}(\bbK)) \hookrightarrow \Hb^{q}(\GL_{2q-4}(\bbK)).
\]
For the general linear groups, we thus obtain bc-stability with slope two.
\end{exmpl}
In Example \ref{GLnExample}, the complexes $X(r)$ can be chosen to be a product complexes, since $\GL_r$ acts essentially $r$-transitively on projective space. It is known that reductive Lie groups whose semisimple part is of type other than $A_n$ do not admit highly transitive actions on generalized flag manifolds; see work by Popov \cite{Popov}. Thus the product version of Theorem \ref{MainThmQual} is insufficient to establish bc-stability for other classical groups. This was our main motivation in stating Theorem \ref{MainThmQual} in its present form.

\begin{exmpl}[Symplectic, orthogonal, and unitary groups] \label{Sp2nExample}
For $\bbK \in \{\bbR,\bbC\}$ and $d \in \bbN$, let $(G_r)_{r\geq 0}$ denote one of the following classical families of Lie groups 
\begin{equation*}
\begin{array}{rrrrrr}
	\{1\} <& \!\! \Sp_2(\bbK) <& \!\!\Sp_4(\bbK) <& \!\! \cdots <& \!\! \Sp_{2r}(\bbK) <& \!\! \cdots \\[2pt]
	\OO_{d,\,0}(\bbK) <& \!\! \OO_{d+1,\,1}(\bbK) <& \!\! \OO_{d+2,\,2}(\bbK) <& \!\! \cdots <& \!\! \OO_{d+r,\,r}(\bbK) <& \!\! \cdots \\[2pt]
	\UU_{d,\,0}(\bbC) <& \!\! \UU_{d+1,\,1}(\bbC) <& \!\! \UU_{d+2,\,2}(\bbC) <& \!\! \cdots <& \!\! \UU_{d+r,\,r}(\bbC) <& \!\! \cdots
\end{array}
\end{equation*}
In our article \cite{DM+Hartnick}, we are going to apply Theorem \ref{MainThmQual} to establish bc-stability for all of these families $(G_r)$. Our proof is based on the fact that all of the groups above arise a 
 automorphism groups of sesquilinear forms. With each of the groups $G_r$ above, one can thus associate a corresponding \emph{Stiefel complex} in the sense of Vogtmann \cite{Vogtmann}, and as in the case of the general linear groups, it is easy to see that the actions are $(r-1)$-transitive and compatible (in the sense of (MQ3c)). In order to establish bc-stability, one thus only has to show that the Stiefel complex $X_r$ is mesurably $\gamma(r)$-acyclic for some proper function $\gamma$. This is established in our article \cite{DM+Hartnick} using probabilistic methods, leading to an acyclicity range of the form $\gamma(r) \sim \log_2(r)$. 
\end{exmpl}
\subsection{The need for finite measurable Quillen families}
Theorem \ref{MainThmQual} does not apply directly to the families of special linear (or special orthogonal or special unitary) groups. The problem is that while $G_r= \SL_r(\bbK)$ acts highly transitively on the product complex $X(r)$ defined as in Example \ref{GLnExample}, the quotient of the stabilizer $H_{r,q}$ by its radical is isomorphic to $\GL_{r-q-1}(\bbK)$ instead of $\SL_{r-q-1}(\bbK)$. 

To deal with this problem, we introduce in Section \ref{SecQuillenFamily} below the notion of a finite measurable Quillen family. It then turns out that (for $\bbK \in \{\bbR, \bbC, \bbH\}$), the family
\[
\{1\} < \GL_1(\bbK) < \GL_2(\bbK) < \dots < \GL_{R-1}(\bbK) < \SL_R(\bbK) 
\]
is a finite measurable Quillen family in this sense, and our main Theorem \ref{MainTheorem} also applies. Assuming the theorem and its notation, we recover Monod's results for $\SL_n$.

\begin{exmpl}[Special linear groups] \label{SLnExample}
Let $\bbK \in \{\bbR,\bbC,\bbH\}$ and $R \geq 2$. We set $G_r := \GL_r(\bbK)$ for any $r \in [R-1]$ and $G_R := \SL_R(\bbK)$. The inclusions $\iota_r:G_r \hookrightarrow G_{r+1}$ are defined as in \eqref{GLn_inclusion} for $r < R-1$, and as
\[
	A \mapsto (\det A)^{-1} \times A
\]
for $r = R-1$. We consider the same acyclic product complexes $X(r)$ as in Example \ref{GLnExample}; then the action of $G_r$ on $X(r)$ is $\tau(r)$-transitive, where $\tau(r) = r$ if $r < R$ and $\tau(R)= R-1$. The group elements $w_{r,q,i} \in G_r$ can be chosen as in Example \ref{GLnExample} for $r <R$, but for $r = R$ one requires a sign adjustment to guarantee that the matrices lie in $G_R$. In the notation of Theorem \ref{MainTheorem} the initial condition is again given by $q_0 = 2$ if $R \geq 3$ and $q_0 = 1$ if $R = 2$. Assume first that $R \geq 3$. Then
\begin{align*}
	\min\{\widetilde{\gamma}(q,R-1), \widetilde{\tau}(q,R-1)-1\} &= \min_{j=q_0}^{q} \big\{\tau\big(R-2(q-j)\big) - j\big\}-1 \\
	&= \min\left\{\min_{j=q_0}^{q-1}\big\{R-2q+j\big\},R-1-q\right\} -1 \\
	&= \min\{R-2q+2,R-q-1\}-1 = R-(2q-2),
\end{align*}
and, hence, the map $\iota_R$ induces
\begin{enumerate}[label=(\roman*)]
\item an isomorphism $\Hb^q(\SL_R(\bbK)) \cong \Hb^q(\GL_{R-1}(\bbK))$ if $R \geq 2q-2$;
\item an injection $\Hb^q(\SL_R(\bbK)) \hookrightarrow \Hb^q(\GL_{R-1}(\bbK))$ if $R \geq 2q-4$.
\end{enumerate}
Combining this information with Example \ref{GLnExample}, we see that for $q \geq 3$ the natural inclusions induce isomorphisms/injections\footnote{Note that the stability range given in \cite{Monod-Stab} is not correct, since it uses the inaccurate stability range for $\GL_r(\bbK)$.}
\[
 \dots \cong \Hb^q(\SL_{2q}(\bbK)) \cong \Hb^q(\SL_{2q-1}(\bbK)) \cong \Hb^q(\SL_{2q-2}(\bbK)) \hookrightarrow \Hb^q(\GL_{2q-3}(\bbK)).\vspace{-2pt}
\]
For $q = 2$, not covered by the discussion above, is well-known (see \cite{Burger-Monod3} and \cite[Ex. 9.9.3]{Monod-Book}):\vspace{-2pt}
\[
 \dots \cong \Hb^2(\SL_{4}(\bbK)) \cong \Hb^2(\SL_{3}(\bbK)) \cong \{0\} \hookrightarrow   \Hb^q(\SL_{2}(\bbK));\vspace{-2pt}
\]
here the latter inclusion is an isomorphism for $\bbK = \bbC$ or $\bbH$, but not for $\bbK = \bbR$. \vspace{-3pt}
\end{exmpl}
{\bf Acknowledgements.} We thank Marc Burger and Nicolas Monod for insightful conversations, and the anonymous referee for valuable feedback on a first version of this article. This work was completed during a postdoctoral fellowship of the first author at the Weizmann Institute of Science, Israel. He was supported by the Swiss National Science Foundation, Grants No. 169106 and 188010. The second author was supported by the Deutsche Forschungsgemeinschaft, Grant No. HA 8094/1-1 within the Schwerpunktprogramm SPP 2026 (Geometry at Infinity).

\section{Spectral sequences associated to Lebesgue complexes} \label{sec:monodcoeff}
Throughout this section, $G$ denotes a locally compact group.  We will assume that $G$ is either second-countable or discrete. The goal of this section is to associate a first-quadrant spectral sequence in continuous bounded cohomology with every partially boundedly acyclic Lebesgue $G$-complex. We explain the terminology used in this statement and recall the necessary background on bounded cohomology, and refer the reader to Monod's book \cite{Monod-Book} for further details. 

\subsection{Coefficient $G$-modules and Lebesgue $G$-spaces}
Following \cite{Monod-Book}, we discuss the class of modules which can be used as coefficients in continuous bounded cohomology. In order to be able to deal with uncountable discrete groups, we have to modify the definitions slightly.

A \emph{Banach $G$-module} is a Banach space over the field of real numbers equipped with a $G$-action by linear isometries, and a \emph{$G$-morphism} is a $G$-equivariant, bounded (but not necessarily isometric) operator between Banach $G$-modules. A Banach $G$-module is said to be \emph{separable} if the underlying Banach space is separable, and \emph{continuous} if the $G$-action is jointly continuous. Given a Banach $G$-module $E$, we denote by $E^\sharp$ the topological dual of $E$, considered as a Banach $G$-module endowed with the contragredient $G$-action. Note that $E^\sharp$ needs not be separable nor continuous, even if $E$ has either of those properties. The following notion is due to Monod \cite{Monod-Book} in the second-countable case:
\begin{defn}\label{DefCoefficientModule}
\begin{enumerate}[label=(\roman*)]
\item If $G$ is non-discrete and second-countable, then a \emph{coefficient $G$-module} is a pair $(E,E^\#)$, where $E$ is a separable, continuous Banach $G$-module.
\item If $G$ is discrete, then a \emph{coefficient $G$-module} is a pair $(E,E^\#)$, where $E$ is a continuous Banach $G$-module (not necessarily separable).
\end{enumerate}
In either case, a \emph{morphism} $(E, E^\#) \to (F, F^\#)$ of coefficient $G$-modules, is a pair $(\Phi,\Phi^\#)$, where $\Phi: E \to F$ is a bounded operator with a $G$-equivariant dual $\Phi^\#: F^\# \to E^\#$. We write ${\bf Coef}_G$ for the category of coefficient $G$-modules.
\end{defn}
\begin{rem}
If we consider a coefficient $G$-module $(E^\flat,E)$, then the only role of the specified pre-dual $E^\flat$  is to define a weak-$\ast$ topology on $E$. We will thus usually drop the pre-dual $E^\flat$ from notation and simply refer to $E$ as a coefficient $G$-module. With this abuse of notation, a morphism between coefficient $G$-modules $E$ and $F$ is simply a $G$-equivariant map $E \to F$, which is continuous with respect to the weak-$*$-topologies defined by the omitted pre-duals.
\end{rem}

The next notion allows us to produce examples of coefficient $G$-modules.

\begin{defn} \label{defn:regspace}
\begin{enumerate}[label=(\roman*)]
\item If $G$ is non-discrete and second-countable, we say that $X$ is a \emph{Lebesgue $G$-space} if it is a standard Borel space endowed with a Borel $G$-action and a $G$-invariant Borel measure class. For economy, we will say that the pair $(X,\mu)$ is a Lebesgue $G$-space whenever $X$ is a Lebesgue $G$-space, and $\mu$ is a fixed $G$-quasi-invariant probability measure within the measure class.
\item If $G$ is discrete, we will say that $X$ is a \emph{Lebesgue $G$-space} if it is a $G$-set, endowed with the measure class of its counting measure.
\end{enumerate}
We denote by ${\bf Leb}_G$ the category whose objects are Lebesgue $G$-spaces and $G$-equivariant, measure-class-preserving Borel maps.
\end{defn}
\begin{exmpl} The space $G^{n}$ is a Lebesgue $G$-space (with respect to the diagonal $G$-action) for every $n \in \mathbb N$. For this to be the case it is crucial that we do not assume a Lebesgue-$G$-space to be countably separated if $G$ is an uncountable discrete group.
\end{exmpl}

\begin{exmpl} \label{LInfty}
Let $X$ be a Lebesgue $G$-space. If $G$ is non-discrete, then we denote by $\mu$ a fixed probability measure in the given measure class on $X$, and if $G$ is discrete, then we denote by $\mu$ the counting measure on $X$. The group $G$ acts on $L^1(X,\mu)$ via
\[
g \cdot \phi(x) := \rho_\mu(g^{-1}, x) \cdot \phi(g^{-1}x),
\]
where $\rho_\mu(g, x) := \dd(g_*\mu)/\dd\!\mu(x)$ denotes the Radon-Nikodym cocycle of $\mu$. In the non-discrete case this action is continuous, as proven by B\"uhler \cite[Appendix D]{Buehler}. Moreover, $L^1(X, \mu)$ is separable in this case (but not necessarily in the discrete case). In either case, the contragredient module is then given by $L^1(X,\mu)^\sharp = L^\infty(X)$ wich action given by $g \cdot f(x) := f(g^{-1}x)$, and our definitions have been set up in such a way that the pair $(L^1(X,\mu), \Linfty(X))$ is a coefficient $G$-module, which up to isomorphism does not depend on the choice of $\mu$. 
\end{exmpl}

\begin{rem}\label{LinftyFunctor}
The assignment $\Linfty(-): {\bf Leb}_G^{\rm op} \to {\bf Coef}_G$ is functorial. In fact, if $T: Y \to X$ is a morphism of Lebesgue $G$-spaces, then the induced operator
\[
	\Linfty(T): \Linfty(X) \to \Linfty(Y), \quad \Linfty(T)(\phi) := \phi \, \circ \,  T
\]
is a morphism of coefficient $G$-modules. Choosing a $G$-quasi-invariant measure $\nu$ on $Y$ as above and letting $\mu:= T_\ast \nu$, the pre-dual operator to $\Linfty(T)$ is
\[
	L^1(T): L^1(Y,\nu) \to L^1(X,\mu), \quad \big( L^1(T)(\psi) \big)(x) := \dd(T_\ast( \psi\!\cdot\!\nu))/\dd\!\mu. 
\]
\end{rem}
We will need the following generalization:
\begin{con}\label{LInftyGeneral}
Let $X$ be a Lebesgue $G$-space and left $E$ be a coefficient $G$-module. We equip $E$ with the Borel $\sigma$-algebra associated with the weak-$*$-topology on $E$. Given $n \in \mathbb N$ we then define
\[
 \Linfty(X;E):=\{\phi: X \to E \: : \: \phi \mbox{ is measurable and essentially bounded}\}/\sim,
\]
where $\sim$ denotes almost everywhere equality. This defines a Banach $G$-module when
equipped with the essential supremum norm and the $G$-action given by the formula $g.f(x) = g.f(g^{-1}x)$. 
In particular, for all $n \in \mathbb N$ we obtain a Banach $G$-module
\[
 \Linfty(G^n;E):=\{\phi: G^n \to E \: : \: \phi \mbox{ is measurable and essentially bounded}\}/\sim.
\]
\end{con}

\begin{rem} Our definition of coefficient $G$-module is made in such a way that if $X$ is a Lebesgue $G$-space and $E$ is a coefficient $G$-module, then $\Linfty(X;E)$ is again a coefficient $G$-module. Indeed, it is the the dual of the Bochner space $L^1(X;E^\flat)$, where $E^\flat$ denotes the underlying pre-dual of $E$. Note that in the discrete case the module $L^1(X;E^\flat)$ need not be separable.

The use of the weak-$*$ Borel structure of $E$ in the definition of $\Linfty(X;E)$ is necessary since $\Linfty(X;E)$ is not well-defined if measurability is taken with respect to its norm-structure. If $G$ is discrete, measurability is a void assumption and $\Linfty(X;E)$ can in fact be defined for \emph{any} Banach $G$-module $E$.

As in the case of the trivial coefficient module $\mathbb R$ one shows that the assignment $X \mapsto \Linfty(X;E)$ is functorial for any fixed coefficient module $E$.
\end{rem}
\subsection{Background on continuous bounded cohomology}
We collect the definition and a series of useful facts on continuous bounded cohomology. 

\begin{defn}
If $E$ is Banach $G$-module and $n \in \mathbb N$, then we denote by $C_{\mathrm{b}}(G^{n+1}; E)$ the Banach space of continuous bounded functions $f: G^{n+1} \to E$ with the supremum norm. The group $G$ acts diagonally on $G^{n+1}$ and hence by isometries on $C_{\mathrm{b}}(G^{n+1}; E)$ via the left-regular representation, and we denote by $C_{\mathrm{b}}(G^{n+1}; E)^G)$ the subspace of $G$-invariant. Then the \emph{continuous bounded cohomology} of $G$ with coefficients in $E$ is defined as 
\[
\Hb^\bullet(G; E) := \HH^\bullet(C_{\mathrm{b}}(G^{\bullet+1}; E)^G).
\]
If $E = \mathbb R$ with trivial $G$-action, we will drop $E$ from notation and write $\Hb^\bullet(G) := \Hb^\bullet(G; \mathbb R)$.
\end{defn}
\begin{prop}\label{MeasRes} If $E$ is a coefficient $G$-module, then
\[
\Hb^\bullet(G; E) = \HH^\bullet(0 \to \Linfty(G; E)^G \to  \Linfty(G^2; E)^G \to  \Linfty(G^3; E)^G \to \dots).
\]
\end{prop}
\begin{proof} In the discrete case, this is just the definition, since $\Linfty(G^{n+1}; E)^G = \ell^\infty(G^{n+1}; E)^G = C_{\mathrm{b}}(G^{n+1}; E)^G$. In the second-countable case, this is a non-trivial result that appears as {\cite[Proposition 7.5.1]{Monod-Book}} in Monod's book.
\end{proof}
To provide a more general class of resolutions, we need to recall the notion of an amenable Lebesgue $G$-space. In the discrete case, its definition is given in \cite[Section 4.9]{Frigerio} in Frigerio's book. In the non-discrete case, the definition is more technical, see \cite[Sec.\ 5.3]{Monod-Book}.
\begin{prop}[{\cite[Thm.\ 7.5.3]{Monod-Book}, \cite[Thm.\ 4.23]{Frigerio}}]\label{MeasRes2} If $E$ is a coefficient $G$-module and $S$ is an amenable Lebesgue $G$-space, then
\[
\Hb^\bullet(G; E) = \HH^\bullet(0 \to \Linfty(S; E)^G \to  \Linfty(S^2; E)^G \to  \Linfty(S^3; E)^G \to \cdots).\vspace{5pt}
\]
\end{prop}

\begin{exmpl}\label{Supergroup} 
The group $G$ acts amenably on $S := G$ via left-multiplication, see \cite[Ex.\ 2.1.2.(i)]{Monod-Book}. If $H<G$ is a closed subgroup and $S$ is an amenable $G$-space, then it is also an amenable $H$-space, see \cite[Ex.~2.1.2.(iv)]{Monod-Book}. In particular, $G$ is an amenable $H$-space, and we deduce from Proposition \ref{MeasRes2} that for any coefficient $H$-module $E$ we have
\begin{equation}\label{ResolutionSupergroup}
	\Hb^\bullet(H;E) = \HH^\bullet(0 \to \Linfty(G;E)^H \to  \Linfty(G^2;E)^H \to  \Linfty(G^3;E)^H \to \cdots).
\end{equation}
\end{exmpl}

The following lemma is an essential technical tool for us. Let ${\bf Vect}$ denote the category of real vector spaces. We say that a sequence 
$
	0 \to E^0 \to E^1 \to E^2 \to \cdots
$ 
of Banach $G$-modules is a \emph{cochain complex}, resp. \emph{exact}, if the underlying sequence of vector spaces has the corresponding property. 

\begin{lem} \label{thm:exactftr}
For any $n \in\bbN$, the functor ${\bf Coef}_G \to {\bf Vect}$, $E \mapsto \Linfty(G^n, E)^G$, is exact .
\end{lem}
\begin{proof}
For discrete groups, this is contained in \cite[Lemma 8.2.4]{Monod-Book}, and, for non-discrete lcsc groups, in \cite[Lemma 8.2.5]{Monod-Book}.
\end{proof}
More concretely, this lemma states that if $0\to A \to B \to C \to 0$ is a short exact sequence of coefficient $G$-modules, then 
\[
	0 \to \Linfty(G^n; A)^G \to \Linfty(G^n; B)^G \to \Linfty(G^n; C)^G \to 0
\]
is an exact sequence of vector spaces.

\subsection{Lebesgue $G$-complexes} \label{sec:GObjects}  

From Lebesgue $G$-spaces, we can build associated complexes using the following notion:
\begin{defn}[{see Weibel \cite[Definition 8.1.9]{Weibel}}] \label{def_ssobjects}
A \emph{semi-simplicial object} $X$ over a concrete category $\clC$ is a sequence of objects $(X_k)_{k \in \bbN}$, together with morphisms $\delta_{i,k}:\, X_{k+1} \to X_k$  for all $k$ and $i \in [k]$, called \emph{face maps}, such that 
\begin{equation} \label{eq:facecondition}
	\delta_{i,k-1} \, \circ \, \delta_{j,k} = \delta_{j-1,k-1} \, \circ \, \delta_{i,k} \qquad \mbox{whenever } i<j\text{.}
\end{equation}
If $X$ and $Y$ are semi-simplicial objects over a category $\clC$, then a \emph{semi-simplicial morphism} $f: X \to Y$ is a collection $(f_k: X_k \to Y_k)_{k \in \bbN}$ of morphisms such that $f_{k} \circ \delta_{i, k-1} = \delta_{i,k} \circ f_{k+1}$ for all $k \in \bbN$ and all $i\in [k+1]$.  
\end{defn}

We will usually omit the index $k$ in $\delta_{i,k}$, writing simply $\delta_{i}$ for all $i$-th face maps. 

\begin{defn}\label{DefGObject} A semi-simplicial object in ${\bf Leb}_G$ will be called a \emph{Lebesgue $G$-complex}. 
\end{defn}

\begin{rem} \label{stabilizer_cont}  If $X$ is a Lebesgue $G$-space, then we refer to points in $\bigsqcup_{k \in \bbN} X_k$ as \emph{simplices} and to points of $X_k$ as \emph{$k$-simplices} of $X$. Given two simplices $x, y \in X$, we say that $x$ is a \emph{face} of $y$, denoted $x \prec y$, if $x$ can be obtained from applying finitely many face maps to $y$. A (finite or infinite) sequence $x_0 \prec x_1 \prec x_2 \prec \cdots$ of simplices is called a \emph{flag}.

If $x,y \in X$ are such that $x \prec y$, then the reversed inclusion $ \stab_G (y) < \stab_G (x)$ of stabilizer subgroups holds, hence every flag $x_0 \prec x_1\prec x_2 \prec \cdots$ gives rise to a chain of closed subgroups
\[
G > \stab_G(x_0) > \stab_G(x_1) > \stab_G(x_2) > \cdots 
\]
The closedness of these inclusions in the case that $G$ is non-discrete follows from \cite[Cor.\ 2.1.20]{Zimmer} in Zimmer's book, since standard Borel spaces are countably separated.
\end{rem}
\begin{con}
Let $X$ be a Lebesgue $G$-complex.  We define $X_{-1}$ to be a singleton, so that $\Linfty(X_{-1}) \cong \bbR$. There is a unique ``face map''
$\delta_{0}: X_0 \to X_{-1}$, which induces the inclusion of constants $\delta^0:\bbR \to \Linfty(X_0)$. 

Applying the $\Linfty$-functor from Remark \ref{LinftyFunctor} to the face maps $\delta_i: X_{l+1} \to X_l$ provides a family of morphisms $\delta^i: \Linfty(X_l) \to \Linfty(X_{l+1})$ of coefficient $G$-modules. If we define
\[
\dd^l: \Linfty(X_l) \to \Linfty(X_{l+1}), \quad \dd^l := {\textstyle \sum_{i=0}^{l+1}} (-1)^i \delta^i \quad (l \geq -1),
\]
then $\dd^{l} \circ \dd^{l-1} = 0$ for all $l \geq 0$, and hence we obtain a cochain complex
\begin{eqnarray}\label{LInftyComplexStiefel}
0 \rightarrow \bbR \xrightarrow{\dd^{-1}} \Linfty(X_0) \xrightarrow{\dd^{0}} \Linfty(X_1) \xrightarrow{\dd^{1}} \Linfty(X_2) \xrightarrow{\dd^{2}} \Linfty(X_3) \, \xrightarrow{} \cdots
\end{eqnarray}
of coefficient $G$-modules.
\end{con}
\begin{defn} \label{def:augmented}
Given a Lebesgue $G$-complex $X$, the cochain complex of coefficient $G$-modules \eqref{LInftyComplexStiefel} is called the \emph{augmented $\Linfty$-complex} associated to $X$.

We will say that $X$ is \emph{boundedly acyclic} if its augmented $\Linfty$-complex is exact, and \emph{boundedly $\gamma_0$-acyclic} for some $\gamma_0 \in \bbN \cup \{\infty\}$ if $\ker \dd^l = \im \dd^{l-1}$ for every $l \in [\gamma_0]$. 

If a Lebesgue $G$-complex $X$ is boundedly $\gamma_0$-acyclic for some $\gamma_0 \in \mathbb N$ that we do not want to specify, then we will say that $X$ is \emph{partially boundedly acyclic}. We will also convene that every Lebesgue $G$-complex is boundedly $(-\infty)$-acyclic. 
\end{defn} 

\subsection{The spectral sequence of a partially boundedly acyclic Lebesgue $G$-complex} \label{sec:specseq_admissible}
With every partially boundedly acyclic Lebesgue $G$-complex, we can now associate a spectral sequence; our notation concerning spectral sequences follows McCleary's book \cite{McCleary}. 

\begin{prop} \label{thm:E_semisimplicial}
	Let $X = (X_q)_{q \in \bbN}$ be a Lebesgue $G$-complex which is boundedly $\gamma_0$-acyclic for some $\gamma_0 \in \mathbb N \cup \{\infty\}$. Then there exists a first-quadrant spectral sequence $\EE_\bullet^{\bullet,\bullet}$ with first page terms and differentials
\begin{equation} \label{eq:firstpage}
	\EE_1^{p,q} = \Hb^q(G;\Linfty(X_{p-1})) \qand \dd_1^{p,q} = \Hb^q(G; \dd^{p-1}) \quad \mbox{for all } p, q \geq 0
\end{equation}
that converges to $\EE_\infty^t = 0$ for all $t \in [\gamma_0+1]$, where $\dd^q$ are the coboundary operators in \eqref{LInftyComplexStiefel}. 
\end{prop}

\begin{proof}
Consider the first-quadrant double complex $(L^{\bullet,\bullet},\dd_{\rm H}, \dd_{\rm V})$, whose terms and differentials for $p,q \geq 0$ are given by
\begin{equation*}
\begin{array}{c}
	L^{p,q}:=\Linfty\big(G^{p+1} \times X_{q-1}\big)^{G} \cong \Linfty\big(G^{p+1}; \Linfty(X_{q-1})\big)^{G}, \vspace{4pt}  \\
\begin{array}{ll}
	\dd^{p,q}_{\rm H}:\,L^{p,q} \to L^{p+1,q}, \qquad &  \dd^{p,q}_{\rm H}\!\! f(g_0,\ldots,g_p):= \sum_{i=0}^p (-1)^i f(g_0,\ldots,\hat{g_i},\ldots,g_p), \vspace{4pt} \\
	\dd_{\rm V}^{p,q}:\,L^{p,q} \rightarrow L^{p,q+1}, \qquad & \dd^{p,q}_{\rm V}\!\! f(g_0,\ldots,g_{p-1}):= \dd^{q-1}\!\big(f(g_0,\ldots,g_{p-1})\big). 
\end{array}
\end{array}
\end{equation*}
\noindent Here $\dd^{q-1}$ denotes the coboundary operator of the augmented $\Linfty$-complex of $X$, as in \eqref{LInftyComplexStiefel}. We let $\IE_\bullet^{\bullet,\bullet}$ and $\IIE_\bullet^{\bullet,\bullet}$ be the two spectral sequences associated with the horizontal and vertical filtrations of $L^{\bullet,\bullet}$, respectively, both of which converge to the cohomology of the total complex of $L^{\bullet,\bullet}$ and whose first-page terms and differentials are given by 
\begin{align*}
\IE_{1}^{p,q} & =\HH^q (L^{p,\bullet},\,\dd_{\rm V}^{p,\bullet}),\quad\Id_{1}^{p,q}=\HH^q(\dd_{\rm H}^{p,\bullet}): \,\IE_{1}^{p,q}\rightarrow\IE_{1}^{p+1,q},\\[-3pt]
\IIE_{1}^{p,q} & =\HH^q(L^{\bullet,p},\,\dd^{\bullet,p}_{\rm H}),\quad\IId_{1}^{p,q}= \HH^q(\dd_{\rm V}^{\bullet,p}):\,\IIE_{1}^{p,q}\rightarrow\IIE_{1}^{p+1,q}. 
\end{align*}
For a proof of the existence and convergence properties of these two spectral sequences, see McCleary \cite[Theorem 2.15]{McCleary}. By the $\gamma_0$-acyclicity of $X$ and exactness of the functor $\Linfty(G^{p+1},-)^G$ for every $p \geq 0$ (see Lemma \ref{thm:exactftr}), the complex
\[
0 \to L^{p,0} \to L^{p,1} \to \cdots \to L^{p,\gamma_0} \to L^{p,\gamma_0+1} \to L^{p,\gamma_0+2} \to \cdots,
\]
is exact up to degree $\gamma_0+1$. Thus, $\IE_1^{p,q} = 0$ for all $p \geq 0$ and $q \in [\gamma_0+1]$, and therefore $\IE_\infty^{t} \cong \bigoplus_{p+q = t} \IE_\infty^{p,q} = 0$  for every $t \in [\gamma_0+1]$. We deduce that also the spectral sequence $\EE_\bullet^{\bullet,\bullet} := \IIE_\bullet^{\bullet,\bullet}$ converges to $0$ in degrees $t \in [\gamma_0 + 1]$. Finally, it follows from Proposition \ref{MeasRes} that the first page of $\EE_\bullet^{\bullet, \bullet}$ is given by \eqref{eq:firstpage}.
\end{proof}

\section{The highly essentially transitive case}
In this section, $G$ continues to denote a locally compact group, either second-countable or discrete, and $X$ denotes a partially boundedly acyclic Lebesgue $G$-complex. In the discrete case, $G$-sets are always endowed with the counting measure. We are going to discuss how certain transitivity properties of the $G$-action on $X$ affect the spectral sequence constructed in the previous section. This will lead us to a proof of Theorem \ref{BasicSS} from the introduction.

\begin{defn} \label{def:conn+trans}
Let $X$ be an Lebesgue $G$-complex and $\tau_0 \in \bbN \cup \{\infty\}$. We say that $X$ is \emph{essentially $\tau_0$-transitive} if $X_l$ admits a conull $G$-orbit for all $l \in [\tau_0]$. We convene that any Lebesgue $G$-complex is essentially $(-\infty)$-transitive. 
\end{defn}  
For the remainder of this section we assume that $X$ is a Lebesgue $G$-complex, which is boundedly $\gamma_0$-acyclic and essentially $\tau_0$-transitive for some $\gamma_0, \tau_0 \in \mathbb N \cup \{\infty\}$. We denote by $\EE_{\bullet}^{\bullet, \bullet}$ the associated spectral sequence constructed in Proposition \ref{thm:E_semisimplicial}.

\begin{rem} The bottom row of the first page of the spectral sequence $\EE_{\bullet}^{\bullet, \bullet}$ is given by
\begin{align*}
\EE_1^{p,0} &= \Hb^0(G;\Linfty(X_{p-1})) = \Linfty(X_{p-1})^G  \qand \\[-3pt]
\dd_1^{p,0} &= \Hb^0(G; \dd^{p-1}) \quad \mbox{for all } p, q \geq 0.
\end{align*}
Since the constant functions are always contained in $\Linfty(X_{p-1})^G$ we have an embedding of constants $c_p: \mathbb R \hookrightarrow \EE_1^{p,0}$ for every $p \geq 0$, such that all of the diagrams
\begin{equation*}
\xymatrixcolsep{2pc}
\xymatrix@R=5pt{\EE_1^{p,0}\ar[rr]^{\delta^i} && \EE_1^{p+1,0}  \\
& \ar@{_{(}->}[ul]^{c_p} \bbR \ar@{^{(}->}[ur]_{c_{p+1}}}
\end{equation*}
commute. Since $G$ acts essentially $\tau_0$-transitive, we have
\[
\EE_1^{p,0} = \Linfty(X_{p-1})^G = c_p(\mathbb R) \cong \bbR \quad \text{for all } p \in [\tau_0+1].
\]
It thus follows from the commuting diagram above that each of the dual face maps $\delta^i: \EE_1^{p,0} \to  \EE_1^{p+1,0}$ equals the identity if 
$p \in [\tau_0]$ and is an injection if $p = \tau_0+1$. We deduce that the face map $\dd_1^{p,0}$ is the zero map if $p \in [\tau_0+1]$ is odd, an isomorphism if $p \in [\tau_0]$ is even, and injective if $p = \tau_0+1$ is even. Consequently we have
\begin{equation}\label{T3}
\EE_2^{p,0} = 0 \; \text{for all }p \in [\tau_0+1].
\end{equation}
\end{rem}
\begin{con}\label{ConGenericFlag} 
By assumption, we can find $o_{\tau_0} \in X_{\tau_0}$ such that the orbit $G.o_{\tau_0} \subset X_{\tau_0}$ is conull. Since the face maps are $G$-equivariant and measure-class-preserving, any codimension-1 face $o_{\tau_0-1} \in X_{\tau_0-1}$ has a conull orbit in $X_{\tau_0-1}$. Thus, we construct a flag $o_{-1} \prec o_0 \prec \cdots \prec o_{\tau_0}$ inductively with $o_q \in X_q$ such that $o_q$ has a conull orbit in $X_q$ for all $q \in \{-1\} \cup [\tau_0]$ (recall our convention that $X_{-1}$ is a singleton, and thus $o_{-1}$ denotes its only element). We refer to such a flag as a \emph{generic flag} in $X$. Given a generic flag, we define 
\[
H_q := \stab_G(o_q) \quad (q \in \{-1\} \cup [\tau_0]);
\]
note that according to our convention, $G$ acts trivially on $X_{-1}$, so $H_{-1} = G$.

By Remark \ref{stabilizer_cont}, we have a chain of closed subgroups
\[
G = H_{-1}>H_0 > H_1 > \dots > H_{\tau_0},
\]
Now let $q \in \{-1\} \cup [\tau_0-1]$ and $i \in [q+1]$, so that $\delta_i(o_{q+1}) \in X_{q}$. By transitivity, there exists $w_{q,i} \in G$ such that
\begin{equation*}\label{wqi}
\delta_i(o_{q+1}) = w_{q,i}^{-1}.o_{q},
\end{equation*}
and we fix a choice of such elements of $G$ once and for all. 
\end{con}

\begin{lem}\label{wLemma}
For all $q \in \{-1\} \cup [\tau_0-1]$ and $i \in [q+1]$, we have the inclusion $w_{q,i} H_{q+1} w_{q,i}^{-1} < H_q$. In particular, there is a well-defined map
\[
\Int(w_{q, i}): H_{q+1} \to H_q, \quad h \mapsto  w_{q,i}hw_{q,i}^{-1}.
\]
\end{lem}
\begin{proof} If $h \in H_{q+1}$, then
\[
w_{q,i}^{-1}.o_q = \delta_i(o_{q+1}) = \delta_i(h.o_{q+1}) = h.\delta_i(o_{q+1}) = h.w_{q,i}^{-1}.o_q \implies w_{q,i}hw_{q,i}^{-1} \in H_q.\qedhere
\]
\end{proof}
\begin{con}
For $p \in [\tau_0+1]$, we define the $(p-1)$-th \emph{induction module} $\Ind^{p-1} := \Linfty(G)^{H_{p-1}}$, where the $H_{p-1}$-invariance is taken with respect to the restriction of the $G$-action by left-translation. Equipped with the right-translation $G$-action $\rho$, the space $\Ind^{p-1}$ is a Banach $G$-module. We now define a family of induction maps
 \begin{equation*} \label{induction_bla}
	\Ind\!\!: \Linfty(G^{q+1})^{H_{p-1}} \to \Linfty(G^{q+1},\Ind^{p-1})^{G}, \quad (\Ind\varphi)(\vec{g})(x)\!:= \varphi(x\vec{g})  \qquad (q \geq 0),
\end{equation*}
which produces a  morphism of cochain complexes (see \cite[Proposition 10.1.3]{Monod-Book}). We also define
\begin{equation*} 
	\Psi_0: \Ind^{p-1} \to \Linfty(X_{p-1}), \quad (\Psi_0\varphi)(g\cdot o_{p-1}) := \varphi(g^{-1}).
\end{equation*}
and define a morphism of cochain complexes $\Psi := \Linfty(G^{q+1};\Psi_0)^{G}$, so that
\begin{equation*} \label{PSI}
	\Psi: \Linfty(G^{q+1};\Ind^{p-1})^{G} \ \to \Linfty(G^{q+1};\Linfty(X_{p-1}))^{G} \quad (q \geq 0).
\end{equation*}
Note that, by Proposition \ref{MeasRes2}, the cohomology of $\Linfty(G^{q+1})^{H_{p-1}}$  is $\Hb^q(H_{p-1})$, and by Proposition \ref{MeasRes}, the cohomology of $ \Linfty(G^{q+1};\Linfty(X_{p-1}))^{G}$ is $\Hb^q(G;\Linfty(X_{p-1})) = \EE_1^{p,q}$.
\end{con}
\begin{prop}\label{propisomstabil} The composition $\Psi \circ \Ind: \Linfty(G^{q+1})^{H_{p-1}} \to \Linfty(G^{q+1};\Linfty(X_{p-1}))^{G}$ induces an isomorphism
\begin{equation*} \label{isom_stabil}
	I_{p,q}:\;\Hb^q(H_{p-1}) \to \EE_1^{p,q} \quad \mbox{for every } p \in [\tau_0+1] \text{ and every } \ q\geq 0. \vspace{-5pt}
\end{equation*}\end{prop}
\begin{proof} It suffices to show that $\Ind$ and $\Psi$ induce isomorphisms in cohomology. The former is just the Eckmann--Shapiro lemma for $\Linfty$-modules (see Monod \cite[Proposition 10.1.3]{Monod-Book} and Monod--Nariman \cite[Lemma 3.5]{Monod-Nariman}) and the latter is immediate from the fact $\Psi_0$ is an isomorphism for any fixed $q\geq 0$ by essential transitivity.
\end{proof}
Having computed the terms on the first page of $\EE_\bullet^{\bullet,\bullet}$, we now compute the differentials. The maps $\Int(w_{p-1, i}): H_p \to H_{p-1}$ from \autoref{wLemma} induce a family of maps $\Hb^q(\Int(w_{p-1, i})): \Hb^q(H_{p-1}) \to \Hb^q(H_p)$, and we claim:
\begin{lem}\label{thm:eckmann_facemap}
For every $p \in [\tau_0]$, $q \geq 0$, and $i \in [p]$, we have a commuting diagram \vspace{-5pt}
\begin{equation*} 
		\xymatrixcolsep{4.5pc}
		\xymatrixrowsep{1pc}
		\xymatrix{\EE_1^{p,q} \ar[r]^{\Hb^q(G;\delta^i)} & \EE_1^{p+1,q} \\
		\Hb^q(H_{p-1}) \ar[u]^{I_{p,q}} \ar[r]^{\Hb^q(\Int(w_{p-1, i}))} & \Hb^q(H_{p}) \ar[u]_{I_{p+1,q}}} 
\end{equation*}
\end{lem}
\begin{proof} For all $i,p,q$ as in the lemma we consider the commutative diagram
\[\begin{xy}\xymatrix{
\bbR \ar[r]\ar[d]_{\mathrm{Id}} &  \Linfty(G) \ar[r] \ar[d]_{\Psi_{i,p,0}} & \Linfty(G^2) \ar[r]\ar[d]_{\Psi_{i,p,1}} & \dots \ar[r] &  \Linfty(G^{q+1}) \ar[r] \ar[d]_{\Psi_{i,p,q}}& \dots\\
\bbR \ar[r] &  \Linfty(G) \ar[r]  & \Linfty(G^2) \ar[r] & \dots \ar[r] &  \Linfty(G^{q+1}) \ar[r] & \dots,
}\end{xy}\]
where $\Psi_{i,p,q}(\varphi)(g_0, \dots, g_q) := \varphi(w_{p-1, i}g_0, \dots, w_{p-1, i}g_q)$. For all $h \in H_{p}$ and $\vec g \in G^{q+1}$:
\begin{align*}
\Psi_{i,p,q}(\Int(w_{p-1,i})(h).\varphi)(\vec g) =& \Int(w_{p-1,i})(h).\varphi(w_{p-1,i}.\vec g) = \varphi(w_{p-1,i} h^{-1}.\vec g)\\ =& \Psi_{i,p,q}(\varphi)(h^{-1}\vec g) = h.\Psi_{i,p,q}(\varphi)(\vec g),
\end{align*}
hence $\Psi_{i,p,q}$ restricts to a map $\Psi_{i,p,q}:  \Linfty(G^{q+1})^{H_{p-1}} \to \Linfty(G^{q+1})^{H_{p}}$ which by \cite[Prop.\ 8.4.2]{Monod-Book} induces the map $\Hb^q(\Int(w_{p-1, i}))$ in cohomology. It thus remains to show only that the diagram
\begin{equation*} 
		\xymatrixcolsep{3pc}	\xymatrix@R=18pt{\Linfty(G^{q+1};\Linfty(X_{p-1}))^{G} \ar[r]^{\delta^i} & \Linfty(G^{q+1};\Linfty(X_p))^{G} \\
		\Linfty(G^{q+1})^{H_{p-1}} \ar[u]^{\Psi \, \circ \Ind} \ar[r]^{\Psi_{i,p,q}} & \Linfty(G^{q+1})^{H_{p}} \ar[u]_{\Psi \, \circ \Ind}} 
\end{equation*}
commutes.  Now if $\varphi \in \Linfty(G^{q+1})^{H_{p-1}}$, $\vec{g} \in G^{q+1}$, and $x \in G$, then
\begin{align*}
	(\delta^i \circ \Psi \circ \Ind)(\varphi)(\vec{g})(x\cdot o_{p}) &= (\Psi \circ \Ind)(\varphi)(\vec{g})(xw_{p-1,i}^{-1} \cdot o_{p-1}) = \varphi(w_{p-1, i} x^{-1} \vec{g})  \\
	&= \Psi_{i,p,q} \varphi(x^{-1}\vec{g}) =  (\Psi \circ \Ind \circ  \Psi_{i,p,q})(\varphi)(\vec{g})(x \cdot o_{p}), 
\end{align*}
and hence $\delta^i \circ \Psi \circ \Ind = \Psi \circ \Ind \circ  \Psi_{i,p,q}$, which proves the lemma.
\end{proof}
Combining this fact with Proposition \ref{thm:E_semisimplicial}, Lemma \ref{wLemma}, Proposition \ref{propisomstabil} and \eqref{T3}, we arrive at Theorem \ref{BasicSS}. 

\section{Measurable Quillen families}\label{SecQuillenFamily}
In Definition \ref{DefInfiniteQuillen}, we have introduced the notion of an infinite measurable Quillen family. Since the quantitative version of Theorem \ref{MainThmQual} (see Theorem \ref{MainTheorem}
below) also applies to certain finite families of groups, we modify the definitions accordingly.
\begin{set}\label{MainSetting} Let $R \in \bbN \cup \{\infty\}$ be a \emph{length parameter} and let \vspace{-3pt}
\[
	\gamma, \ \tau : [R] \to \bbN \cup \{\pm\infty\} \vspace{-3pt}
\]
be two functions, called respectively the \emph{acyclicity range} and \emph{transitivity range}. We assume that we are given
\begin{itemize}
\item a locally compact group $G_r$ for every $r \in [R]$, all either second-countable or discrete; 
\item an embedding $\iota_r: G_r \hookrightarrow G_{r+1}$ for every $r < R$;
\item  a Lebesgue $G_r$-complex $X(r)$ for every $r \in [R]$.
\end{itemize}
We set $G_r := \{1\}$ for all $r < 0$. Furthermore, we keep our convention from last section that $X(r)_{-1}$ is a singleton, and we denote by $o_{r,-1}$ its only element. 
\end{set}
Generalizing Definition \ref{DefInfiniteQuillen}, we declare:
\begin{defn} We say that $(G_r,X(r))_{r \in [R]}$ is a \emph{measurable $(R,\gamma, \tau)$-Quillen family}  provided 
 \begin{enumerate}[leftmargin=37pt, label=(MQ\arabic*)]
	\item $X(r)$ is a boundedly $\gamma(r)$-acyclic Lebesgue $G_r$-complex for every $r \in [R]$;
	\item $X(r)$ is essentially $\tau(r)$-transitive Lebesgue $G_r$-complex for every $r \in [R]$;
	\item for every $r\in [R]$, there exists a generic flag  $o_{r,-1} \prec o_{r,0} \prec o_{r,1} \prec \cdots \prec o_{r,\tau(r)}$ in $X(r)$ such that the corresponding stabilizers $H_{r,q} := \stab_{G_r}(o_{r,q}) < G_r$ are compatible with $G_{r-\tau(r)}, \dots, G_{r}$ either in the sense of Condition (MQ3a), (MQ3b) or (MQ3c) from the introduction for elements $w_{r,q,i} \in G_r$ with $r \in [R], \, q \in [\tau(r)-1],$ and $i \in [q+1]$.
\end{enumerate} 
In particular, a measurable $(\infty, \gamma, \tau)$-Quillen family is the same as a measurable $(\gamma, \tau)$-Quillen in the sense of Definition \ref{DefInfiniteQuillen}. 
\end{defn}
Applying Theorem \ref{BasicSS} to each of the complexes $X(r)$ in a measurable Quillen family, we obtain:
\begin{cor}\label{SpectralSequencesMain} Assume that $(G_r,X(r))_{r \in [R]}$ is a measurable  $(R,\gamma, \tau)$-Quillen family. Then for every $r \in [R]$, there exists a spectral sequence ${}^r\!\EE_\bullet^{\bullet,\bullet}$ such that
\begin{enumerate}[label = \emph{(\roman*)},leftmargin=25pt]
\item ${}^r\!\EE_\bullet^{\bullet,\bullet}$ converges and ${}^r\!\EE_\infty^t = 0$ for every $t \in [\gamma(r)+1]$.
\item ${}^r\!\EE_1^{p,q} = \Hb^q(G_{r-p})$ for all $p \in [\tau(r)+1]$ and $q\geq 0$.
\item ${}^r\!\EE_2^{p,0} = 0$ for all $p \in [\tau(r)+1]$ and $q \geq 0$. 
\item For all $p \in [\tau(r)]$ and $q\geq 0$, the map $\dd_1^{p,q}:  \Hb^q(G_{r-p}) \to \Hb^q(G_{r-p-1})$ is given by
\[
\dd_1^{p,q} = \left\{\!\!\begin{array}{ll}
	\Hb^q(\iota_{r-p-1}) & \mbox{if } p \mbox{ is even,} \\
	0 & \mbox{if } p \mbox{ is odd.}
\end{array}\right.
\]
\end{enumerate}
\end{cor}
\begin{proof}
Since the inflation maps in bounded cohomology whose inducing epimorphisms have amenable kernels are isomorphisms (see Monod \cite[Corollary 8.5.2]{Monod-Book}), Condition (MQ3) produces the following commutative diagram for all $r \in [R]$, $p \in [\tau(r)]$, $i \in [p]$ and $q \geq 0$: \vspace{-3pt}
\begin{equation*} \label{projection_cohom}
		\begin{gathered}	
		\xymatrixcolsep{5.5pc}
		\xymatrixrowsep{1.5pc}
		\xymatrix{\Hb^q(H_{r,p-1}) \ar[r]^{\Hb^q({\rm Int}(w_{r,p-1,i}))} & \Hb^q(H_{r,p})  \\
		\Hb^q(G_{r-p}) \ar[u]_{\cong}^{\Hb^q(\pi_{r,p})} \ar[r]^{\Hb^q(\iota_{r-p-1})} & \Hb^q(G_{r-p-1}) \ar[u]_{\Hb^q(\pi_{r,p-1})}^{\cong}} 
		\end{gathered} \vspace{-3pt}
\end{equation*} 
The corollary follows now from Theorem \ref{BasicSS}.
\end{proof}

If for any $s \in [R-1]$ and any $q \geq 0$ we abbreviate \vspace{-1.5pt}
\[
\HH^{q}_{s} := \Hb^{q}(G_{s}) \qand \iota_{s}^q:= \Hb^q(\iota_{s}), \vspace{-1.5pt}
\]
then the first page of ${}^r\!\EE^{\bullet, \bullet}_1$ is given by Figure \ref{fig:sampleE}. 
\begin{figure}[t!]
\[\resizebox{\hsize}{!}{
\xymatrix@R=1pt@C=6pt{
 & q & & && && && && && && && & \\
 \\
 & & \vdots && \vdots && \vdots && \vdots && \vdots && \vdots && \vdots && \vdots && \vdots \\
{\textstyle q} & & \HH_r^q \ar[rr]^-{\iota_{r-1}^q} && \HH_{r-1}^q \ar[rr]^-{0} && \HH_{r-2}^q \ar[rr]^-{\iota_{r-3}^q} && \HH_{r-3}^q \ar[rr]^-{0} && \ \cdots \ \ar[rr]^-{\color{blue} 0}_-{\color{red} \iota_{r-\tau(r)+1}^q} && \HH_{r-\tau(r)+1}^q \ar[rr]^-{\color{blue} \iota_{r-\tau(r)}^q}_-{\color{red} 0} && \HH_{r-\tau(r)}^q \ \ar[rr]^-{\color{blue} 0}_-{\color{red} \iota_{r-\tau(r)-1}^q} && \HH_{r-\tau(r)-1}^q \ar[rr]^{\!\!\!\!\!?} && {}^r\!\EE_1^{\tau(r)+2,q} \ar[r] & \cdots \\
 & & \vdots && \vdots && \vdots && \vdots && \vdots && \vdots && \vdots && \vdots && \vdots \\ \\
\ar[rrrrrrrrrrrrrrrrrrrr] & & && && && && && && && && && & \hspace{-8pt} p \\
& \ar[uuuuuuu] & {\scriptstyle 0} && {\scriptstyle 1} && {\scriptstyle 2} && {\scriptstyle 3} && {\scriptstyle \cdots} && {\scriptstyle \tau(r)-1} && {\scriptstyle \tau(r)} && {\scriptstyle \tau(r)+1} & \ar@{..}[uuuuuuu] & {\scriptstyle \tau(r)+2} & \cdots
}}
\]	
\vspace{-10pt}

\caption{First page ${}^r\!\EE_1^{\bullet,\bullet}$.}
\label{fig:sampleE}
\end{figure}
The terms to the left of the dotted line are explained by Corollary \ref{SpectralSequencesMain} (ii), while the behavior of the terms to the right of the dotted line is \emph{a priori} not understood. According to Corollary \ref{SpectralSequencesMain} (iii), the parity of $\tau(r)$ determines the pattern of the arrows just to the left of the dotted line: the blue labels correspond to the case in which $\tau(r)$ is odd, and the red ones underneath correspond to the even case. 
\begin{rem}  Let $(G_r,X(r))_{r \in [R]}$ be a measurable  $(R, \gamma, \tau)$-Quillen family. We say that $q_0 \in \mathbb N$ is an \emph{initial parameter} for this measurable Quillen family if
\[
 \Hb^q(\iota_r):  \Hb^q(G_{r+1}) \to \Hb^q(G_r) \text{ is an isomorphism for all }q\leq q_0 \text{ and }r<R. 
\]
In theory, one would always like to work with the optimal initial parameter
\[
q_0 := \sup\{q \mid \Hb^q(\iota_r): \Hb^q(G_{r+1}) \to \Hb^q(G_r) \text{ is an isomorphism for all }r<R\},
\]
but in practice, the optimal initial parameter is not known. If no \emph{a priori} infomation about the bounded cohomology of the group $G_r$ is available, then one can always work with $q_0 := 1$, since $\Hb^1(G)$ vanishes for lcsc groups $G$. In some cases of interest, better initial parameters are available. Notably, if all $G_r$ are connected semisimple Lie groups with finite center and either all of the groups $G_r$ are of Hermitian type or all are of non-Hermitian type, then $q_0 \geq 2$. 
\end{rem}
To state the quantitative version of our bc-stability criterion, we introduce the following
\begin{defn} Let $(G_r,X(r))_{r \in [R]}$ be a measurable  $(R, \gamma, \tau)$-Quillen family with initial parameter $q_0 \in \mathbb N$. We define the associated \emph{dual acyclicity range} and \emph{dual transitivity function} to be functions $\widetilde{\gamma},\widetilde{\tau}:\,\bbN\times [R-1] \to \bbN \cup \{\pm\infty\}$ defined by the formulae 
\begin{equation*}
	\widetilde{\gamma}(q, r) := \min_{j=q_0}^q \big\{\gamma\big(r+1-2(q-j)\big) - j\big\} \qand \widetilde{\tau}(q,r) :=  \min_{j=q_0}^q \big\{\tau\big(r+1-2(q-j)\big) - j\big\}.\vspace{3pt}
\end{equation*}
for $r + 1 - 2(q-q_0) \geq 0$, and $\widetilde\gamma(q,r) = \widetilde\tau(q,r) = -\infty$ otherwise. 
\end{defn}
We state now our main result:
\begin{thm}\label{MainTheorem} Let $(G_r,X(r))_{r \in [R]}$ be a measurable  $(R, \gamma, \tau)$-Quillen family with initial parameter $q_0 \geq 1$ and associated dual acyclicity range $\widetilde{\gamma}$ and dual transitivity function $\widetilde{\tau}$. Then the inclusion $\iota_r$ induces an isomorphism (resp. an injection) 
\[\Hb^q(\iota_r): \Hb^q(G_{r+1}) \, \xrightarrow{\cong \ } \, \Hb^q(G_r) \quad (\text{resp. }\Hb^{q+1}(\iota_r): \Hb^{q+1}(G_{r+1}) \hookrightarrow \Hb^{q+1}(G_r)).
\] 
provided $r \in [R-1]$ and $q \geq 0$ satisfy the condition
\begin{equation}
\min\{\widetilde{\gamma}(q,r), \widetilde{\tau}(q,r)-1\} \geq 0.
\end{equation}
\end{thm}

An immediate corollary of this quantitative bc-stability result is Theorem \ref{MainThmQual}, as stated in the introduction.
\begin{proof}[Proof of Theorem \ref{MainThmQual} from Theorem \ref{MainTheorem}] We may assume without loss of generality that $q_0 = 1$. Since $\gamma(r) \to \infty$ and $\tau(r) \to \infty$, we find for every $q \geq q_0$ some $r(q) \in \mathbb N$ such that for all $j \in \{q_0, \ldots, q\}$ and all $r \geq r(q)$, we have
\[
 \gamma\big(r+1-2(q-j)\big) \geq  j \qand  \tau\big(r+1-2(q-j)\big) \geq j+1.
\]
Then Theorem \ref{MainTheorem} implies the chain of isomorphisms \vspace{-3pt}
\[
	\Hb^q(G_{r(q)}) \xleftarrow{\ \cong} \, \Hb^q(G_{r(q)+1}) \xleftarrow{\ \cong} \, \Hb^q(G_{r(q)+2}) \xleftarrow{\ \cong} \, \cdots,
\] 
which is the desired bc-stability.
\end{proof}
The functions $\widetilde{\gamma}(q,r)$ and $\widetilde{\tau}(q,r)$ have been defined so that they satisfy the following lemma, which will allow us to make inductive arguments:
\begin{lem}\label{Combinatorics} Let $r,\, q \in \bbN$, $q \geq q_0$, and assume that $\min\{\widetilde{\gamma}(q,r), \widetilde{\tau}(q,r)-1\} \geq 0$. Then the following hold:
\begin{enumerate}[label = \emph{(\roman*)},leftmargin=25pt]
\item $\gamma(r+1) \geq q$ and $\tau(r+1) \geq q+1$. 
\item $\min\{\widetilde{\gamma}(q-1,r),\widetilde{\tau}(q-1,r)-1\} \geq 0$. 
\item If $p \in [q-q_0]$ is odd, then $\min\{\widetilde{\gamma}(q-p, r-p-1),\widetilde{\tau}(q-p, r-p-1)-1\} \geq 0$. 
\item If $p \in [q-q_0]$ is even, then $\min\{\widetilde{\gamma}(q-p-1,r-p-2),\widetilde{\tau}(q-p-1,r-p-2)-1\} \geq 0$. 
\end{enumerate}
\end{lem}
\begin{proof} The assumption means that for all $j \in \{q_0, \dots, q\}$ we have
\begin{align} 
	r+1-2(q-j) &\geq 0, \qand \label{eq:cond_isom_welldef} \\[-3pt]
	\gamma\big(r+1-2(q-j)\big) &\geq j, \quad \tau\big(r+1-2(q-j)\big) \geq j + 1.\label{eq:cond_isom_bis} \vspace{-5pt}
\end{align}
\begin{enumerate}[label=(\roman*),leftmargin=25pt]
\item Choose $j=q$ in \eqref{eq:cond_isom_bis}.
\item If $q=q_0$, then $r+1-2(q-1-q_0) \geq 0$ and $\widetilde\gamma(q-1,r) = \widetilde\tau(q-1,r) = \infty$ for any $r$. Assume now that $q > q_0$, and let $j \in \{q_0,\ldots,q-1\}$. Observe that 
\[
	r+1-2(q-1-j) = r+1-2\big(q-(j+1)\big). \vspace{-1pt}
\]
Since $j+1 \in \{q_0,\ldots,q\}$, from this equality combined with \eqref{eq:cond_isom_welldef} and \eqref{eq:cond_isom_bis} we deduce  \vspace{-1pt}
\begin{align*}
	\qquad r+1-2(q-1-j) &\geq 0, \qand \\[-3pt]
	\qquad \gamma\big(r+1-2(q-1-j)\big) &\geq j+1 > j, \quad \tau\big(r+1-2(q-1-j)\big) \geq j+2 > j + 1,
\end{align*} 
which establishes the claim. 

\item Observe that for any odd $p \in [q-q_0]$ and any $j \in \{q_0,\ldots,q-p\}$, we have 
\[
	r-p - 2(q-p-j) = r+1 -2 \left(q-\left(j+\frac{p-1}2\right)\right), 
\]
and $j+(p-1)/2 \in \{q_0 + (p-1)/2,\ldots,q-(p+1)/2\} \subset \{q_0,\ldots,q\}$. Just as in (ii), the claim follows now from \eqref{eq:cond_isom_welldef} and \eqref{eq:cond_isom_bis}.

\item 
For any even $p \in [q-q_0]$ and any $j \in \{q,\ldots,q-p\}$, we have
\[
	r-p-1-2(q-p-1-j) = r+1-2\left(q-\left(j+\frac{p}2\right)\right),
\]
and $j + p/2 \!\in\! \{q_0 + p/2,\ldots,q-p/2\} \!\subset\! \{q_0,\ldots,q\}$, and conclude by \eqref{eq:cond_isom_welldef} and \eqref{eq:cond_isom_bis}. \qedhere
\end{enumerate}
\end{proof}
\section{Proof of the main theorem}
The remainder of this article is devoted to deriving Theorem \ref{MainTheorem} from Corollary \ref{SpectralSequencesMain}. We consider a measurable  $(R, \gamma, \tau)$-Quillen family $(G_r,X(r))_{r \in [R]}$ with initial parameter $q_0 \geq 1$, dual acyclicity range $\widetilde{\gamma}$ and dual transitivity function $\widetilde{\tau}$ and abbreviate
\[
\HH^{q}_{s} := \Hb^{q}(G_{s}) \qand \iota_{s}^q:= \Hb^q(\iota_{s})
\]
as before. We consider the following statement for $q \in \mathbb N$ and $r < R$:
\begin{itemize}
\item[{($S_{q,r}$)}] $\iota^q_r: \HH^q_{r+1} \to \HH^q_r$ is an isomorphism and $\iota^{q+1}_{r}: \HH^{q+1}_{r+1} \to \HH^{q+1}_r$ is an injection.
\end{itemize}
The conclusion of Theorem \ref{MainTheorem} then says that for all $q \in \bbN$, the following statement ($S_q$) holds true:
\begin{itemize}
\item[{($S_q$)}] If $r<R$ satisfies $\min\{\widetilde{\gamma}(q,r), \widetilde{\tau}(q,r)-1\} \geq 0$, then ($S_{q,r}$) holds.
\end{itemize}
If $q < q_0$, then $r + 1 - 2(q-q_0) \geq 0$ and hence by definition $\widetilde{\gamma}(q,r) = \widetilde{\tau}(q,r) = \infty$. Since $q_0$ is an initial parameter, both $\iota^q_r$ and $\iota^{q+1}_r$ are isomorphisms for all $r < R$. Thus ($S_q$) holds for all $q < q_0$. We now proceed to prove ($S_q$) for $q \geq q_0$ by induction on $q$. For this purpose, assume that $q \geq q_0$ and that ($S_{q'}$) holds for all $q' < q$. In view of Lemma \ref{Combinatorics}, we have:
\begin{lem}\label{ABC} For $r < R$ with $\min\{\widetilde{\gamma}(q,r), \widetilde{\tau}(q,r)-1\} \geq 0$, the following hold:
\begin{enumerate}[label=\emph{(\Alph*)}]
	\item The map $\iota_{r}^q$ is injective. 
	\item The map $\iota^{q-p}_{r-p-1}$ is an isomorphism for all odd numbers $p \in [q]$.
	\item The map $\iota_{r-p-2}^{q-p}$ is an injection for all even numbers $p \in [q]$.
\end{enumerate} 
\end{lem}

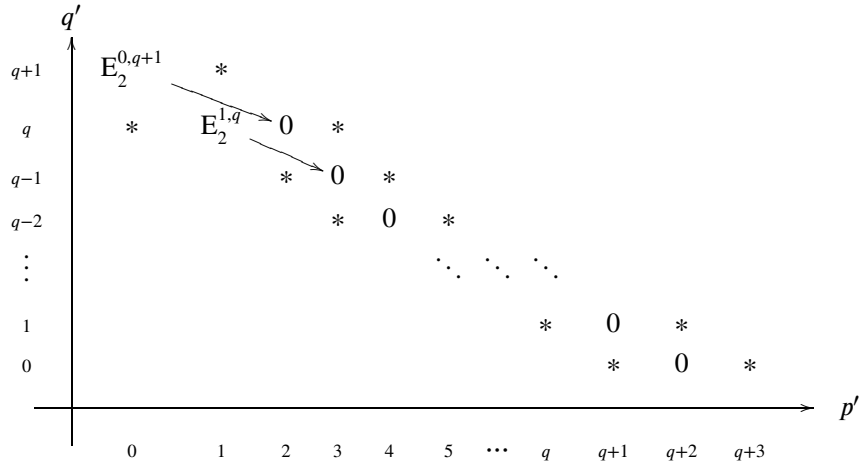
\begin{figure}[b!]
\hspace{-25pt} 
\[\resizebox{0.75\hsize}{!}{
\xymatrixrowsep{0.5pc}
\xymatrix@R=2pt@C=1pt{
 & q' \\
{\scriptstyle q+1} & & \EE_2^{0,q+1} \ar[rrrrd] && \ast && && && && && && && && \\
{\scriptstyle q} & & \ast && \EE_2^{1,q} \ar[rrrrd] && 0 && \ast && && && && && && \\
{\scriptstyle q-1} & & && && \ast && 0 && \ast && && && && && \\
{\scriptstyle q-2} & & && && && \ast && 0 && \ast && && && && \\ 
\vdots & & && && && && && \ddots &\ddots & \ddots \\ 
{\scriptstyle 1} & & && && && && && && \ast && 0 && \ast && \\ 
{\scriptstyle 0} & & && && && && && && && \ast && 0 && \ast \\ 
\ar[rrrrrrrrrrrrrrrrrrrrrrr] & & && && && && && && && && && && & \hspace{-8pt} & p' \\
& \ar[uuuuuuuuu] & {\scriptstyle 0} && {\scriptstyle 1} && {\scriptstyle 2} && {\scriptstyle 3} &&  {\scriptstyle 4} && {\scriptstyle 5} & \cdots & {\scriptstyle q} && {\scriptstyle q+1} && {\scriptstyle q+2} && {\scriptstyle q+3} 
}
}\]
\caption{Second page $\EE_2^{\bullet,\bullet}$}
\label{fig:IIE_2}
\end{figure}

\begin{proof} (A) By Lemma \ref{Combinatorics} (ii), the inequality $\min\{\widetilde{\gamma}(q-1,r),\widetilde{\tau}(q-1,r)-1\} \geq 0$ holds. Thus, we may apply ($S_{q-1}$) to deduce the claim.

(B) If $p > q - q_0$, then $q-p < q_0$ and hence $\iota^{q-p}_{r-p-1}$ is an isomorphism by the initial condition. If $p \leq q-q_0$, then by Lemma \ref{Combinatorics} (iii) we have $\min\{\widetilde{\gamma}(q-p, r-p-1),\widetilde{\tau}(q-p, r-p-1)-1\} \geq 0$. We may thus use ($S_{q-p}$) to show that ($S_{q-p, r-p-1}$) is true, implying the claim.

(C) The case $p > q-q_0$ is again covered by the initial condition, and if $p \leq q-q_0$, then by Lemma \ref{Combinatorics} (iv) we have $\min\{\widetilde{\gamma}(q-p-1,r-p-2),\widetilde{\tau}(q-p-1,r-p-2)-1\} \geq 0$. Thus, by ($S_{q-p-1}$), the statement ($S_{q-p-1, r-p-2}$) applies.
\end{proof}
To establish ($S_q$), we fix a natural number $r < R$ such that $\min\{\widetilde{\gamma}(q,r), \widetilde{\tau}(q,r)-1\} \geq 0$ and consider the spectral sequence $\EE_\bullet^{\bullet,\bullet} := {}^{r+1}\!\EE_\bullet^{\bullet,\bullet}$  from Corollary \ref{SpectralSequencesMain}. By Lemma \ref{Combinatorics} (i), we have
\begin{equation}\label{D}
	\gamma(r+1) \geq q \qand \tau(r+1) \geq q+1,
\end{equation} 
and, thus, Corollary \ref{SpectralSequencesMain} gives
\begin{equation} \label{eq:Gr+1}
\begin{gathered}
\begin{array}{lll}
	\EE_1^{p',q'} \!\!\!&= \HH^{q'}_{r+1-p'} &  \mbox{ for } p' \in [q+2],\, q' \geq 0, \\[2pt]
	\EE_2^{p', 0} \!\!\!&= 0 & \mbox{ for } p' \in [q+2],\\
	\dd_1^{p',q'} \!\!\!&= \left\{\!\!\begin{array}{ll}
	\iota^{q'}_{r-p'} & \mbox{if } p' \mbox{ is even,} \\
	0 & \mbox{if } p' \mbox{ is odd,}
\end{array}\right.  &  \mbox{ for } p' \in [q+1],\, q' \geq 0,
\\
	\EE_\infty^{t} \!\!\!&= 0 & \mbox{ for } t \in [q + 1].
\end{array}
\end{gathered}
\end{equation}
The second equality is \eqref{T3}. We are going to show now that
\begin{equation}\label{T1}
\EE_2^{0, q+1} \cong \ker(\iota^{q+1}_r) \qand 
\EE_2^{1,q} \cong \coker(\iota^{q}_{r}), \vspace{-13pt}
\end{equation}
\begin{equation}\label{T2}
\EE_2^{p', q'} = 0 \text{ for all }(p',q') \in \{0,\dots, q\}^2 \text{ with }p'+q' = q+2.
\end{equation}

These computations will be sufficient to establish ($S_q$): As indicated in Figure \ref{fig:IIE_2}, the differentials emanating from $\EE_2^{0, q+1}$ and $\EE_2^{1,q}$ as well as all those emanating from the positions $(0,q+1)$ and $(1,q)$ on any of the following pages end up on the diagonal $p'+q' = q+2$. Since this diagonal contains only zeros up to row $q$, we conclude that 
 \[
\EE_\infty^{0, q+1} \cong \EE_2^{0, q+1} \cong \ker(\iota^{q+1}_r)
\qand \EE_\infty^{1,q} \cong \EE_2^{1,q} \cong \coker(\iota^{q}_{r}).
\]
In particular, it follows that 
$
\ker(\iota^{q+1}_r) \oplus \coker(\iota^{q}_{r}) \hookrightarrow E_\infty^{q+1} = 0.
$
This gives the injectivity of $\iota_r^{q+1}$ and the surjectivity of $\iota_r^q$. Since $\iota_r^q$ is injective by (A), the claim ($S_q$) is proven. We have thus reduced the proof of Theorem \ref{MainTheorem} to \eqref{T1} and \eqref{T2}.

To prove \eqref{T1} and \eqref{T2}, we only need to consider the few arrows on the first page of the spectral sequence depicted in Figure \ref{fig:IIE_1}. 
The statement \eqref{T1} is immediate from the fact that the leftmost maps in the $q$-th and $(q+1)$-th row of $\EE_1^{\bullet, \bullet}$ are respectively given by
\[
\HH^q_{r+1} \xrightarrow{\iota_r^q} \HH^q_{r} 
\qand \HH^{q+1}_{r+1} \xrightarrow{\iota^{q+1}_r} \HH^{q+1}_{r} 
\]

As for \eqref{T2}, we first consider $q' \in \{1, \ldots, q\}$ and define $p' := q-q' \in [q-1]$ so that $p'+q' = q$. By \eqref{eq:Gr+1}, the two maps
$
\EE_1^{p'+1,q'} \to \EE_1^{p'+2,q'}  \to \EE_1^{p'+3,q'}  
$
are given by
\[
\left\{ \begin{array}{ll}
	 \HH^{q-p'}_{r-p'} \xrightarrow{0} \HH^{q-p'}_{r-p'-1} \xrightarrow{\iota^{q-p'}_{r-p'-2}} \HH^{q-p'}_{r-p'-2}, & \text{if }p' \text{ is even,} \\
	 \HH^{q-p'}_{r-p'} \xrightarrow{\iota^{q-p'}_{r-p'-1}} \HH^{q-p'}_{r-p'-1} \xrightarrow{0} \HH^{q-p'}_{r-p'-2}, & \text{if }p' \text{ is odd.}  
\end{array} \right.
\]
From (B) and (C) of Lemma \ref{ABC}, we deduce that 
\[
\EE_2^{p'+2, q'} = 0 \quad \text{for all } (p', q') \in [q]^2  \text{ with }p'+q' = q,
\]
where the case $(p',q')=(q,0)$ is covered in \eqref{eq:Gr+1}. This concludes the proof of \eqref{T2}. \qed
\begin{figure}[t!]
\[\resizebox{0.75\hsize}{!}{
\xymatrix@R=1pt@C=0.5pt{
 & q' \\
{\scriptstyle q+1} & & \HH^{q+1}_{r+1} \ar[rr] && \HH^{q+1}_{r} && && && && && && && && \\
{\scriptstyle q} & & \HH^{q}_{r+1}  \ar[rr] && \HH^{q}_{r}  \ar[rr] && \HH^{q}_{r-1} \ar[rr] && \HH^{q}_{r-2} && && && && && && \\
{\scriptstyle q-1} & & && && \ar[rr] \HH^{q-1}_{r-1} \ar[rr] && \HH^{q-1}_{r-2} \ar[rr] && \HH^{q-1}_{r-3} && && && && && \\
{\scriptstyle q-2} & & && && && \HH^{q-2}_{r-2} \ar[rr] && \HH^{q-2}_{r-3} \ar[rr] && \HH^{q-2}_{r-4} && && && && \\
\vdots & & && && && && && \ddots & \ddots & \ddots \\
{\scriptstyle 1} & & && && && && && && \HH^{1}_{r-q+1} \ar[rr] && \HH^{1}_{r-q} \ar[rr] && \HH^{1}_{r-q-1} && \\  
{\scriptstyle 0} & & && && && && && && && \HH^0_{r-q} \ar[rr] && \HH^0_{r-q-1} \ar[rr] && E_1^{q+3,0} \\ 
\ar[rrrrrrrrrrrrrrrrrrrrrrr] & & && && && && && && && && && && && & \hspace{-8pt} & p' \\
& \ar[uuuuuuuuu] & {\scriptstyle 0} && {\scriptstyle 1} && {\scriptstyle 2} && {\scriptstyle 3} &&  {\scriptstyle 4} && {\scriptstyle 5} & \cdots & {\scriptstyle q} && {\scriptstyle q+1} && {\scriptstyle q+2} && {\scriptstyle q+3}
}}
\]	
\vspace{-10pt}
\caption{First page ${}^{r+1}\EE_1^{\bullet,\bullet}$.}
\label{fig:IIE_1}
\end{figure}


\begin{thebibliography}{123}
\bibitem{Bestvina}
	M. Bestvina.
	\emph{Homological stability of Aut($F_n$) revisited}. 
	Hyperbolic geometry and geometric group theory, 1-11,
	Adv. Stud. Pure Math. 73, 
	Math. Soc. Japan, 
	Tokyo, 2017. 
	
\bibitem{BBI}
	M. Bucher, M. Burger, A. Iozzi. 
	\emph{The bounded Borel class and 3-manifold groups}.
	Duke Math. J. 167, 
	no. 17 (2018), 3129--3169.
	
\bibitem{Buehler}
	T. B\"uhler.
	\emph{On the algebraic foundations of bounded cohomology}.  
	Mem. Amer. Math. Soc. 214 (2011), no. 1006, xxii+97 pp.
	
\bibitem{Burger-Monod1}
	M. Burger, N. Monod.
	\emph{Bounded cohomology of lattices in higher rank Lie groups.} 
	J. Eur. Math. Soc. (JEMS) 1 (1999), no. 2, 199-235. 
	
\bibitem{Burger-Monod3}
	M. Burger, N. Monod.
	\emph{On and Around the Bounded Cohomology of $\SL_2$}. In: M. Burger, A. Iozzi, A. (eds) Rigidity in Dynamics and Geometry. Springer, Berlin, Heidelberg, 2002. 

\bibitem{DM-Thesis}
	C. De la Cruz Mengual. 
	\emph{On Bounded-Cohomological Stability for Classical Groups.} 
	ETH Zurich (2019), Zurich.

\bibitem{DM-Complex}
	C. De la Cruz Mengual.
	\emph{The Degree-Three Bounded Cohomology of Complex Lie Groups of Classical Type.}
	arXiv preprint (2023), \href{https://arxiv.org/abs/2304.00607}{arXiv:2304.00607}.
	
\bibitem{DM+Hartnick}
	C. De la Cruz Mengual, T. Hartnick.
	\emph{Stabilization of Bounded Cohomology for Classical Groups.}
	arXiv preprint (2022), \href{https://arxiv.org/abs/2201.03879}{arXiv:2201.03879}.
	
\bibitem{Essert}
	J. Essert.
	\emph{Homological stability for classical groups}. 
	Israel J. Math. 198 (2013), no. 1, 169-204. 
	
\bibitem{Frigerio}
	R. Frigerio.
	\emph{Bounded cohomology of discrete groups}.
	Mathematical Surveys and Monographs, 227. 
	American Mathematical Society, 
	Providence, RI, 2017. 

\bibitem{Gromov}
	M. Gromov.
	\emph{Volume and bounded cohomology}.
	Inst. Hautes \'Etudes Sci. Publ. Math. No. 56 (1982), 5-99 (1983). 

\bibitem{Harer}
	J.L. Harer.
	\emph{Stability of the homology of the mapping class groups of orientable surfaces}.
	Ann. of Math. (2) 121 (1985), no. 2, 215-249.
	
\bibitem{HatVogt}
	A. Hatcher, K. Vogtmann.
	\emph{Homology stability for outer automorphism groups of free groups}. 
	Algebr. Geom. Topol. 4 (2004), 1253-1272. 
	
\bibitem{vdK}
	W. van der Kallen.
	\emph{Homology stability for linear groups}.
	Invent. Math. 60 (1980), no. 3, 269-295. 
	
\bibitem{Quillen}
	D. Quillen.
	\emph{Notebook 1974-1}. 
	Available at \href{http://www.claymath.org/publications/quillen-notebooks}{http://www.claymath.org/publications/quillen-notebooks}.

\bibitem{McCleary}
	J. McCleary.
	\emph{A user's guide to spectral sequences}. 
	Second edition. Cambridge Studies in Advanced Mathematics, 58.
	Cambridge University Press, Cambridge, 2001. 

\bibitem{Monod-Book} 
	N. Monod.
	\emph{Continuous bounded cohomology of locally compact groups}. 
	Lecture Notes in Mathematics, 1758. 
	Springer-Verlag, 
	Berlin, 2001. 
	
\bibitem{Monod-Stab} 
	N. Monod.
	\emph{Stabilization for $\SL_n$ in bounded cohomology}. 
	Discrete geometric analysis, 191-202, Contemp. Math., 347, 
	Amer. Math. Soc., 
	Providence, RI, 2004. 
	
\bibitem{Monod-Vanish} 
	N. Monod.
	\emph{Vanishing up to the rank in bounded cohomology}. 
	Math. Res. Lett. 14 (2007), 
	no. 4, 681-687. 
	
\bibitem{Monod-Nariman}
	N. Monod, S. Nariman.
	\emph{Bounded and unbounded cohomology of homeomorphism and diffeomorphism groups}
	Invent. Math. 232 (2023), 
	no.3, 1439--1475.
	
\bibitem{Popov}
	V.L. Popov.	
	\emph{Generically multiple transitive algebraic group actions}. 
	Algebraic groups and homogeneous spaces, 481-523,
	Tata Inst. Fund. Res. Stud. Math., 19, Tata Inst. Fund. Res., Mumbai, 2007. 
	
\bibitem{Sprehn-Wahl2}
	D. Sprehn, N. Wahl.
	\emph{Homological stability for classical groups.}
	Trans. Amer. Math. Soc. 373 (2020), no. 7, 4807-4861.

\bibitem{Vogtmann}
	K. Vogtmann.
	\emph{A Stiefel complex for the orthogonal group of a field.}
	Comment. Math. Helv. 57 (1982), no. 1, 11-21.

\bibitem{Weibel}
	C.A. Weibel.
	\emph{An introduction to homological algebra}. 
	Cambridge Studies in Advanced Mathematics, 38. Cambridge University Press, Cambridge, 1994.
	
\bibitem{Zimmer}
	R. J. Zimmer.
	\emph{Ergodic theory and semisimple groups}.
	Monographs in Mathematics, Vol. 81. Boston-Basel-Stuttgart: Birkh\"auser, 1984.

\end{thebibliography}
\end{document}